\PassOptionsToPackage{unicode}{hyperref}
\PassOptionsToPackage{hyphens}{url}

\documentclass[11pt
]{article}
\usepackage{amsmath}
\usepackage{bbm}
\usepackage{verbatim}
\usepackage{amssymb}
\usepackage{amsthm}
\usepackage{lmodern}
\usepackage{iftex}
\ifPDFTeX
  \usepackage[T1]{fontenc}
  \usepackage[utf8]{inputenc}
  \usepackage{textcomp} 
\usepackage{fancyhdr}
\fi

\usepackage[ 
  sorting=nyt         
]{biblatex}
\IfFileExists{upquote.sty}{\usepackage{upquote}}{}
\IfFileExists{microtype.sty}{
  \usepackage[]{microtype}
  \UseMicrotypeSet[protrusion]{basicmath}
}{}
\usepackage{xcolor}
\IfFileExists{xurl.sty}{\usepackage{xurl}}{} 
\IfFileExists{bookmark.sty}{\usepackage{bookmark}}{\usepackage{hyperref}}

\usepackage{graphicx}
\makeatletter
\def\maxwidth{\ifdim\Gin@nat@width>\linewidth\linewidth\else\Gin@nat@width\fi}
\def\maxheight{\ifdim\Gin@nat@height>\textheight\textheight\else\Gin@nat@height\fi}
\makeatother

\setkeys{Gin}{width=\maxwidth,height=\maxheight,keepaspectratio}

\makeatletter
\def\fps@figure{htbp}
\makeatother
\usepackage[normalem]{ulem}

\ifLuaTeX
  \usepackage{selnolig}  
\fi

\usepackage[margin=3cm]{geometry}

\newcommand{\Q}{\ensuremath{\mathbf{Q}}}
\newcommand{\R}{\ensuremath{\mathbf{R}}}
\newcommand{\C}{\ensuremath{\mathbf{C}}}
\newcommand{\Z}{\ensuremath{\mathbf{Z}}}
\newcommand{\A}{\ensuremath{\mathbf{A}}}
\newcommand{\F}{\ensuremath{\mathbf{F}}}
\newcommand{\on}[1]{\operatorname{#1}}
\newcommand{\ord}{\on{ord}}
\newcommand{\bpm}{\begin{pmatrix}}
\newcommand{\epm}{\end{pmatrix}}
\newcommand{\gm}{\gamma}
\newcommand{\bs}{\backslash}
\newcommand{\us}{\underline{\sigma}}
\newcommand{\bbm}{\begin{bmatrix}}
\newcommand{\ebm}{\end{bmatrix}}
\newcommand{\bsbm}{\left[\begin{smallmatrix}}
\newcommand{\esbm}{\end{smallmatrix}\right]}
\newcommand{\ol}{\overline}
\newcommand{\ul}{\underline}
\newcommand{\f}{\mathfrak}

\newtheorem{thm}[equation]{Theorem}
\newtheorem{cor}[equation]{Corollary}
\newtheorem{prop}[equation]{Proposition}
\newtheorem{lem}[equation]{Lemma}
\newtheorem{rmk}[equation]{Remark}
\newtheorem{rmks}[equation]{Remarks}

\newtheorem{conj}[equation]{Conjecture}

\usepackage{tikz}

\title{On a theorem of Jiang and Rallis}
\usepackage{etoolbox}
\makeatletter
\providecommand{\subtitle}[1]{
  \apptocmd{\@title}{\par {\large #1 \par}}{}{}
}
\makeatother

\author{\underline{Yaniel Rivera Vega}\\ 
\emph{Department of Risk $\&$ Insurance, University of Wisconsin - Madison, WI
} \\ \emph{riveravega@wisc.edu} \\
\vspace{2ex} \\ 
Joseph Hundley, Ph.D \\ 
\emph{Department of Mathematics, University at Buffalo, NY}\\ \emph{jahundle@buffalo.edu} \\
\vspace{2ex} \\ 
with an appendix by \\
Victor Scharaschkin\\
\emph{vscharaschkin@gmail.com}
}
\date{\vspace{-3ex}}

\begin{document}
\maketitle

\section{Introduction}
In modern number theory, a lot of research is dedicated to the 
study of zeta functions and their analytic properties. 
The first and simplest example of such a function, and 
possibly still the most important, is the Riemann zeta function, 
$\zeta(s),$ defined for $s=\sigma+it \in \C$ with 
$\sigma > 1$ by the absolutely convergent Dirichlet series 
$$
\zeta(s) = \sum_{n=1}^\infty \frac1{n^s},
$$
or by the absolutely convergent Euler product (over primes $p$)
$$
\prod_p \left(1- p^{-s}\right)^{-1},$$
but extending to an analytic function on $\C - \{ 1\},$
and satisfying the functional equation 
$$
\pi^{-s/2} \Gamma(\frac s2) \zeta(s) = \pi^{\frac{s-1}2}\Gamma( \frac{1-s}{2}) \zeta(1-s),
$$
where $\Gamma$ is the Gamma function, given, for $\Re(s) > 0$ 
by 
$$\Gamma(s) = \int_0^{\infty} e^{-u} u^{s-1} \, du.$$

This idea was extended to number fields (i.e., finite extensions
of $\Q$) by Dedekind, who defined a zeta function attached to 
such a field 
$$
\zeta_K(s) = \sum_{\f a}N\f a^{-s} = \prod_{\f p}(1-N \f p^{-s})^{-1}, 
$$
where the sum is over all ideals in the so-called ring of
integers, $\f o_K,$ of $K,$ the product is over just the prime ideals, 
and for each ideal $\f a,$ the norm $N\f a$ is just the index 
of $\f a$ in $\f o_K,$ which is finite. 

The same idea was extended in a different direction by Dirichlet, 
who introduced Dirichlet characters $\chi: \Z \to \C,$ which satisfy 
$$
\chi(mn) = \chi(m) \chi(n), \qquad \chi(m+N) = \chi(m), \qquad
\chi(n) = 0 \iff \gcd(n,N) \ne 1,
$$
for some integer $N.$ 
Dirichlet studied the $L$ function 
$$
L(s, \chi) = \sum_{n=1}^\infty \frac{\chi(n)}{n^{-s}}
= \prod_p (1-\chi(p) p^{-s})^{-1}.
$$

The two extensions are related by quadratic reciprocity. Indeed, 
if $K$ is any quadratic extension of $\Q,$ then there is a 
unique Dirichlet character $\chi,$ which takes values in 
$\{-1,0,1\},$ such that 
\begin{equation}\label{eq: zF=zL}
\zeta_K(s) = \zeta(s) L(s, \chi).\end{equation}
In more detail, for each prime $p,$ one of three things 
happens: 
\begin{enumerate}
    \item There are two prime ideals $\f p_1$ and $\f p_2$ in 
    $\f o_K$ such that $N\f p_1= N\f p_2 = p,$ and 
    $\chi(p) = 1;$ in this case we get a factor of 
    $(1-p^{-s})^{-2}$ in the Euler products 
    on both sides of \eqref{eq: zF=zL}.
    \item There is one prime ideal $\f p$ in $\f o_K$
    such that $N\f p = p^2,$ and $\chi(p) = -1;$ in this case we 
    get a factor of $(1-p^{-2s})^{-1} = (1-p^{-s})^{-1}(1+p^{-s})^{-1}$
    in the Euler products 
    on both sides of \eqref{eq: zF=zL}.
    \item There is one prime ideal $\f p$ in $\f o_K$
    such that $N\f p = p,$ and $\chi(p) = 0;$ in this case we 
    get a factor of $(1-p^{-s})^{-1} = (1-p^{-s})(1-0\cdot p^{-s})$
    in the Euler products 
    on both sides of \eqref{eq: zF=zL}.
\end{enumerate}
For example, if $K=\Q[i]$ then $\f o_K= \Z[i],$ which is 
a p.i.d., and $\chi$ is given by 
$$
\chi(n) = \begin{cases} 1, & n \equiv 1 \mod 4, \\ 
-1, & n \equiv 3 \mod 4, \\ 
0 , & n \text{ is even.}\end{cases}
$$
There is an analogue of the identity \eqref{eq: zF=zL} 
in the setting where $F$ is a
quadratic extension, not of $\Q,$ but of another 
field $F$ which is itself a finite extension of 
$\Q.$ In this case $\zeta_K/\zeta_F$ is attached to 
a Hecke character, which is like a Dirichlet 
character, but defined on $\f o_F$ rather than $\Z.$
These and related facts are covered in most introductory 
algebraic number theory books, such as
\cite{Lang},\cite{CassFroh}.

There is also an analogue of the identity \eqref{eq: zF=zL} in
the setting where 
the extension is  of degree $>2,$ but in this case 
$\zeta_K/\zeta$ (or, more generally, $\zeta_K/\zeta_F$)
is not a simple Dirichlet (or, more generally, Hecke) $L$ 
function, but a more mysterious type of Euler product, 
known as an Artin $L$ function. Unlike Dirichlet and Hecke 
$L$ functions, Artin $L$ functions do not have analytic 
properties that are well understood. 

In the paper 
\cite{Jiang-Rallis}, Jiang and Rallis introduced a new 
way to study the ratio $\zeta_K/\zeta_F$ in the case 
when $[K:F]=3.$ The method is based on an 
integral over a group of matrices with entries in a 
ring called the adeles. The adele ring of a number
field $F$ is a natural object to use in the 
study of Euler products, because the adele ring itself 
is a type of product, with a factor for 
each prime, and a finite number of other factors. 
For example, in the case of $\Q$ there is one additional 
factor, and it corresponds to factor of 
$\pi^{-s/2} \Gamma(\frac{s}2)$ which must be added to the 
original Euler product for $\zeta(s)$ in order to obtain a 
nice functional equation. 
We will describe the 
adele ring of $\Q$ in more detail below. 

The integral $I^{\sigma}(f_s,s)$ 
of Jiang and Rallis depends on an element 
$\sigma$ of $F^4.$ This quadruple determines a polynomial $g$
of degree at most three, and when $g$ is cubic and  irreducible, that provides the 
cubic extension of $F.$
In their paper, Jiang and Rallis 
prove that the integral has nice analytic properties by 
relating it to a function called an Eisenstein series 
with known analytic properties. This Eisenstein series is 
defined on a group of matrices, with entries in the adele 
ring, which is known as $G_2.$

Having established the analytic properties of their integral, 
Jiang and Rallis set out to 
calculate it. When $\sigma$ determines a cubic extension $K,$
the goal is to 
prove that the integral is equal to the product 
of $\zeta_K(3s-1)/\zeta_F(3s-1),$
and an additional factor, known as the 
``normalizing factor'' of the Eisenstein series 
mentioned above, and given by 
$$
\frac{1}{\zeta_F(3s)\zeta_F(6s-2)\zeta_F(9s-3)}.
$$
(They are also able to predict what the value 
of $I^{\sigma}(f_s, s)$
should be when the polynomial attached to $\sigma$ 
is reducible, with distinct roots.)
Using the product structure of the adele ring, Jiang and 
Rallis factor their integral as a product over primes, 
with a factor $I^{\sigma}_{\f p}(f_{s,\f p}, s)$ for each 
prime ideal $\f p$ in $\f o_F$ and a finite number 
of other factors. 
The 
job is then to match the contribution from each prime with the 
corresponding factor in the Euler product for  
$$
\frac{\zeta_K(3s-1)}{\zeta_F(3s)\zeta_F(3s-1)\zeta_F(6s-2)\zeta_F(9s-3)}.
$$ Jiang and Rallis were able to 
accomplish this goal only under two additional hypotheses:
\begin{itemize}
    \item they restrict attention to primes which are ``unramified.'' In essence this means that they assume that anything which only happens 
    at a finite number of primes is {\it not} happening at the prime 
    they are considering. (For example, if $F$ is $\Q,$ we might exclude 
    any prime which divides the denominator of one of the elements
    of $\sigma$: there are only finitely many such primes.)
    \item they assume that the field 
$F$ contains three cube roots of $1.$
\end{itemize}
Note that this second technical hypothesis  unfortunately excludes the case $F=\Q$!

Jiang and Rallis comment that the restriction on cube roots 
should be removable.
A first step towards removing it was taken in the 
master's thesis of Joseph Pleso \cite{pleso2009integrals}.
The hypothesis that $F$ contains three cube roots of $1$
is used
in several different ways in \cite{Jiang-Rallis}, including in an 
argument which reduces the study of arbitrary irreducible cubic
polynomials to the study of those of the form 
$x^3-a.$ Jiang and Rallis break the integral 
$I_{\f p}^{\sigma}(f_{s, \f p}, s)$
into 16 pieces, and the
complexity of these pieces is greater when there are more nonzero 
coefficients in the polynomial attached to $\sigma.$ In 
\cite{pleso2009integrals}, Pleso considers a polynomial of 
the form $x^2-bx+c,$ and computes the analogous 16 
sub-integrals. He then evaluates nine of them. 

In this paper, we restrict ourselves to the important special case $F=\Q.$ In this case $\f o_F=\Z.$
So, if $\f p$ is a prime ideal of $\f o_F,$ then $\f p$ is generated by an ordinary prime $p.$ 
For each prime $p,$ we obtain the field $\Q_p$ of $p$-adic numbers, which we discuss in more 
detail in the next section. The Jiang-Rallis integral $I_p(f_{s,p}^\circ, s)$ is an integral
over five copies of $\Q_p,$ and is also described in more detail in the next section. 
If $p \equiv 1\pmod{6},$ then the field $\Q_p$ has three cube roots of $1.$ In this case, the value of $I_p(f_{s,p}^\circ, s)$ can be deduced from 
the results of Jiang-Rallis, even though the 
field $\Q$ does not have three cube roots of $1.$ 

Thus one, only has to handle $p=2, p=3$ and $p \equiv 5 \pmod{6}.$ In this paper, we 
let $p \equiv 5 \pmod{6},$
compute six of the seven sub-integrals which Pleso did not compute, and
reduce the last one to a conjecture about the number of
integral solutions to a polynomial equation over $\Z/p\Z.$

\begin{conj}\label{conj0}
Let $p$ be a prime, such that $p \equiv 5 \pmod{6}.$ 
Fix $b$ and $c$ in $\Z/p \Z^\times$ such that 
$g(u) := -u^3 +bu +c$ is irreducible. Then 
$$\{(r,y,u) \in (\mathbf{Z}/p\mathbf{Z})^* \times (\mathbf{Z}/p\mathbf{Z}) \times (\mathbf{Z}/p\mathbf{Z}) \mid   g(u) \equiv 3uyr + y^3r^2 -r \mod p\}$$
has $p^2-1$ elements.
\end{conj}

This conjecture has since been proved by Victor Scharaschkin, who has kindly permitted us to include his
proof in an appendix to this paper.
Our main theorem can be stated as follows. 
\begin{thm}
Let $I^{\sigma}_p(f_{s,p}^\circ, s)$ be the local integral of Jiang 
and Rallis at an unramified prime $p.$ Assume that 
$p \equiv 5 \pmod{6},$ that the polynomial 
$g$ attached to $\sigma$
is irreducible mod $p.$ 
Then 
$$I^{\sigma}_p(f_{s,p}^\circ, s)
=(1-p^{-3s})(1-p^{-3s+1})(1-p^{-6s+2}).
$$
\end{thm}

The assumption that the polynomial $g$ is irreducible
mod $p$ is a natural one. If $g$ is reducible in the field $\Q_p$, then 
the integral has already been computed by Jiang and Rallis, who did not 
use their hypothesis on the cube roots of 1 in their treatment of this
case. And, for any polynomial $f$ with integer coefficients, the set 
of primes $p$ such that $f$ is reducible mod $p$ but not in $\Q_p$
is finite. 

So, by assuming that  $p \equiv 5\pmod{6}$ and
and that $g$ is irreducible,
we exclude only the cases already handled by
Jiang-Rallis and a finite number of exceptions.

The right-hand side of our identity is the expected one. Indeed, 
when the polynomial attached to $\sigma$ is irreducible mod $p$, 
and $K$ is the cubic extension attached to it, there is 
a single prime ideal $\f p$ in $\f o_F$ such that $N\f p=p^3.$
Thus, the contribution to the Euler product for $\zeta_F(3s-1)$
corresponding to the prime $p$ is $(1-(p^3)^{-3s+1})^{-1},$
or $(1-p^{-9s+3}),$ and so the contribution of the prime $p$ to 
the Euler product for 
$$
\frac{\zeta_K(3s-1)}{\zeta(3s)\zeta(3s-1)\zeta(6s-2)\zeta(9s-3)}.
$$
is 
$$
\frac{(1-p^{-9s+3})^{-1}}
{(1-p^{-3s})^{-1}(1-p^{-3s+1})^{-1}(1-p^{-6s+2})^{-1}(1-p^{-9s+3})^{-1}},
$$
which is precisely the right-hand side of our identity.

The integral
$I_{\f p}^{\sigma}(f_{s, \f p}^\circ , s)$
was previously computed
by Xiong in \cite{Xiong}
using 
 a totally different method which works 
 for any number field $F.$ Thus, our  
 ``Conjecture \ref{conj0}'' may actually be deduced 
 from Xiong's result, using the results of this paper. 
Scharaschkin's elementary proof of conjecture \ref{conj0}  completes a second, more elementary proof of 
the special case $F=\Q$ of Xiong's result.

\section{Background and Notation}
\subsection{The field of \texorpdfstring{$p$}{p}-adic numbers} 
In this section we review some basic definitions about the field
$\Q_p$ of $p$-adic numbes, where $p\in \Z$ is a prime number. 
Some references for a more detailed introduction to $\Q_p$  are
\cite{Katok}, \cite{Sally-FundMathAn},\cite{Sally-LettMathPhys},\cite{Sally-ToolsOfTrade}, \cite{Goldfeld-Hundleyv1}.

\subsubsection{Analytic approach}
The first, more analytic approach to defining $\Q_p$ is based 
on defining a non-standard absolute value on the field $\Q$ of 
rational numbers, which depends on a prime $p.$
First, let $\on{ord}_p: \Q^{\times} \to \Z$
be the unique function such that, for each element $x$ of $\Q^\times,$ we have
$x = p^{\ord_p(x)}\frac ab$ for some $a, b \in \Z,$ neither of which 
is divisible by $p.$ It is conventional to set the $\on{ord}_p(0) = +\infty$ to ensure
$\on{ord}_p$ satisfies all the axioms of a valuation. We then define the $p$-adic absolute value $\Q \to [0,\infty)$ by 
$$|x|_p = \left\{
   \begin{array}{cc}
      \frac{1}{p^{\ord_p(x)}} & \text{if } x \neq 0,\\
      0, & \text{if } x = 0. \\
   \end{array}\right.$$
This formula does indeed give a well-defined absolute value on $\Q,$
and indeed, one which satisfies the ``strong triangle inequality'': 
$$|a+b|_p \le \max( |a|_p, |b|_p)$$
with equality whenever $|a|_p \ne |b|_p.$
(Equivalently, $\ord_p(a+b) \ge \min( \ord_p(a), \ord_p(b)),$ with equality whenever
$\ord_p(a) \ne \ord_p(b).$)
We may the define the field 
$\mathbf{Q}_p$ of $p$-adic numbers as the
completion of $\Q$ with respect to $|\ |_p$; we define 
``Cauchy'' and ``convergent'' using $|\ |_p$, show that having a difference which converges to $0$ is an equivalence relation on 
Cauchy sequences, and let $\Q_p$ be the 
set of equivalence classes of Cauchy sequences, relative to this
equivalence relation. Thus defined, $\Q_p$ is indeed a field, 
and addition, subtraction, multiplication, etc., as well as $|\ |_p$
have unique continuous extensions from $\Q$ to $\Q_p$ which we 
denote with the same symbols.

The ring of $\mathbf{Z}_p$ $p$-adic integers may then be defined as $\{x \in \mathbf{Q}_p : |x| \leq 1 \}.$  This is indeed a subring 
of $\Q_p,$ and it is also the closure of $\Z$ in the topology 
on $\Q_p,$ and is compact in this topology. 
This ring has a unique maximal, ideal, namely $\{x \in \mathbf{Q}_p : |x| < 1 \},$ which is principal and generated by $p$; its group of units $\mathbf{Z}_p^*$ is thus $\{x \in \mathbf{Q}_p : |x| = 1 \}.$
Finally, $\mathbf{Q}_p - \mathbf{Z}_p = \{x \in \mathbf{Q}_p : |x| > 1 \}.$ 

The field $\Q_p$ possesses a measure, unique up to multiplication by 
a positive scalar, which assigns open sets positive measure, 
assigns compact sets finite measure, and is invariant under additive 
translation. We can pin it down uniquely by specifying that the 
measure of $\Z_p$ is $1.$ This measure is called the additive Haar
measure and we denote it $dx$ (or $dy$ if the variable is $y,$ etc).

\subsubsection{Algebraic Approach}
For each positive integer $k,$ the set $\{x \in \Z_p: |x| \le p^{-k} \}$ is the principal ideal $p^k\Z_p$ in $\Z_p$ generated by $p^k.$ It satisfies $\Z_p/p^k\Z_p \cong \Z/p^k\Z.$ This actually permits
one to give an equivalent, algebraic definition of $\Z_p$ and $\Q_p.$
Specifically, one may define 
$\Z_p$ as the projective limit $\varprojlim_k \Z/p^k \Z$ of 
the finite rings $\Z/p^k \Z, \ k = 1,2,3 \dots,$ i.e., the set of 
infinite sequences $(a_k)_{k=1}^\infty$ such that $a_k \in \Z/p^k\Z$ 
and if $\ell > k$ then $a_k$ is the image of $a_\ell$ under the 
``reduction mod $p^k$'' map $\Z/p^\ell \Z \to \Z/p^k \Z.$
If this alternate definition of $\Z_p$ is used, then $\Q_p$ may be 
recovered as the field of fractions of $\Z_p.$

Via either approach, each nonzero element of $\Q_p$ has a unique
$p$-adic expansion: 
$$\sum_{n=N}^\infty d_n p^n, \qquad  N \in \Z, \qquad d_n \in \{0, \dots, p-1\}, \text{ each }n, \qquad d_N \ne 0.$$
Then $\sum_{n=N}^\infty d_n p^n$ lies in $\Z_p$ if and only if 
$N \ge 0.$

\subsubsection{More on Haar measure}\label{ss:Haar}

As we have mentioned, the Haar measure is invariant under 
additive translation and assigns the ring $\Z_p$ a measure of $1.$
For each integer $k$ the set 
$$p^k \Z_p = \{ p^kx : x \in \Z_p\}$$
is at once an additive subgroup of $\Q_p,$ the closed ball
of radius $p^{-k}$ centered at $0,$
and the open ball of radius $p^{1-k}$ centered at $0.$
For each other $a\in \Q_p$ the coset $a+p^k\Z_p$ is 
 the closed ball
of radius $p^{-k}$ centered at $a,$
and the open ball of radius $p^{1-k}$ centered at $a.$

By invariance, the various cosets $a+p^k\Z_p,$ as $a$ varies, 
all have the same measure. If $k > 0$ then 
$\Z_p$ 
is a union of $p^{k}$ cosets of $p^k \Z_p.$
If $k < 0$ then $p^k \Z_p$ is a union of $p^{-k}$ cosets of $\Z_p.$
In either case we can deduce that the Haar measure
of $a+p^k \Z_p$ is $p^{-k}$ for all $k \in \Z.$

The cosets $a+p^k\Z_p, \ (a \in \Q_p, k \in \Z)$ play a role in the 
theory of Haar measure on $\Q_p$ that is similar to the role 
played by intervals in the theory of Lebesgue measure 
on $\R.$ For example 
the ``step functions'' used to define the 
integral are finite linear combinations 
of characteristic functions of these balls:
\begin{equation}
    \label{eq:stepFunction}
\sum_{i=1}^n c_i \mathbbm{1}_{a_i+p^{k_i}  \Z_p}.
\end{equation}
(Here, and throughout this paper, 
$\mathbbm{1}_X$ is the characteristic function of the set $X.$)
Their integrals are completely determined by the volumes of these balls, which we just computed. 

It follows that if $X,$ and $Y$ are two Haar-measurable subsets of $\Q_p$
and $\varphi: X \to Y$ is a function which maps each coset
$a+p^k \Z_p$ to another coset $b+p^k \Z_p$ of the same size
(i.e., with the same $k$), then $\varphi$ preserves the Haar measure, 
so that 
$$
\int_X h(\varphi(x))\, dx = \int_Y h(y) \, dy,
$$
for any measurable function $h.$ 

If $X$ and $Y$ happen to be subsets of $\Z_p,$ then the cosets
$a +p^k \Z_p$ are precisely the elements of the quotient ring 
$\Z/p^k \Z,$ so the measure-preserving property can be interpreted
as giving a well-defined bijection modulo $p^k$ for each $k.$ 

In what follows, we shall make use not only of Haar measure on $\Q_p$ 
but also the product measure defined on several copies of $\Q_p.$
In this context we shall frequently refer to the measure of a
measurable set $X$ as its volume, and denote it $\on{Vol}(X).$

\subsection{Hensel's Lemma}
\begin{lem}
For 
$$f(x) = \sum_{i=0}^d c_i x^i \in \Z_p[x],$$
let 
$$f'(x) =\sum_{i=1}^d i c_i x^{i-1} $$
be the derivative. For $a_0 \in \Z_p,$ suppose that 
$$f(a_0) \equiv 0 \pmod{p}, \qquad \text{and} \qquad f'(a_0) \not \equiv 0 \pmod{p}.$$
Then 
$$\{ a \in \Z_p: a \equiv a_0 \pmod{p}, \text{ and } f(a) = 0 \}$$
has exactly one element. 
\end{lem}
\begin{proof}
    See \cite{Katok}, theorem 1.39.
\end{proof}
\begin{cor}
    Take $f(x)\in \Z_p[x]$ and let $\ol f(x) \in \Z/p\Z[x]$ be the polynomial obtained by reducing each of the coefficients of $f$ modulo $p.$ Suppose that $f$ is irreducible, while $\ol f$ vanishes at $a \in \Z/p\Z.$ Then $a$ is actually a
    double zero of $\ol f.$
\end{cor}
\subsection{Additive characters}
The function 
$$e_p(x) = \begin{cases}
\exp\left(-2\pi i\sum_{n=N}^{-1} d_n p^n\right), 
& x = \sum_{n=N}^\infty d_n p^n,\\
1,& x \in \Z_p,
\end{cases}
$$
is a well-defined continuous homomorphism $\Q_p\to S^1:= \{ z \in \C: |z|=1\}.$

It should be noted that if $F$ is a number field -- that is, a 
finite extension of $\Q$ -- then it is possible to define 
absolute values which are analogous to $|\ |_p$ on the field 
$F$ and complete $F$ with respect to these absolute values
and obtain fields which share many of their properties with 
$\Q_p.$ 
These fields are known as non-Archimedean local fields.
Notably, many of the results of Jiang and Rallis are 
proved for any number field or for any non-Archimedean local 
field. Others are proved for any field in one of these classes which contains three cube roots of one-- a technical 
hypothesis which excludes $\Q,$ and excludes $\Q_p$ if $p \equiv 2 \pmod{3}.$

\subsection{The adele ring} 
\label{ss:adele ring}
In this section we briefly describe the so-called adele ring of $\Q.$
It's worth noting that this theory may be extended to define adele rings
attached to finite extensions of $\Q$ as well. Some good references 
are 
\cite{Lang}, \cite{CassFroh}, and \cite{Goldfeld-Hundleyv1}. 

In order to set up the definition, it is useful to mention an important
fact. It follows from a theorem of 
Ostrowski (see \cite{CassFroh}, p. 45) that every nontrivial 
absolute value 
on $\Q$ is either the standard one (which gives $\R$ as the 
completion of $\Q$), or the absolute value $|\ |_p$ for some 
prime $p$, or equivalent to $|\ |_p$ for some $p$ in the sense that
it induces the same topology on $\Q.$ (Concretely, if we take $c \in (1, \infty)$, which is not equal to $p$, then replacing ``$p^{-\ord_p(x)}$'' by ``$c^{-\ord_p(x)}$'' in the definition of $|\ |_p$ produces an absolute value which is not the same as $|\ |_p,$ but 
is equivalent to it.)

Thus, the set of topologically-distinct completions of $\Q$  
consists of $\R$ and one element for each prime $p.$ 
It is convenient to introduce an index set which consists of the 
primes and one additional element corresponding to $\R.$
Historically, the most common choice for the index corresponding 
to the completion $\R$ is ``$\infty.$'' In keeping with this
tradition, we let $\Q_\infty = \R$ and $|\ |_\infty: \R \to [0, \infty)$
be the usual absolute value. 

We now consider the product of all the topologically-distinct completions of $\Q$:
$$\R \times \prod_p \Q_p.$$
An element of this ring can be thought of as an infinite sequence 
$x = (x_\infty, x_2, x_3, x_5, x_7, \dots, x_p, \dots )$
such that $x_\infty \in \R$ while $x_p \in \Q_p$ for each prime $p.$
The adele ring, $\A$ is then 
$$
\{ x \in \R \times \prod_p \Q_p: \text{ the set of primes } p \text{ such that } x_p \notin \Z_p \text{ is finite}\}.
$$

Notice that if $a$ is an element of $\Q$ then $a$ is an element of 
every completion of $\Q$ and hence $(a,a,a,a,a, \dots, a, \dots )$
is an element of $\R \times \prod_p \Q_p.$ In fact, it is an element 
of $\A$ because $a \in \Z_p$ for any prime $p$ which 
does not divide the denominator of $a$ when written in lowest terms.
 We regard $\Q$ as a 
subring of $\A$ by identifying $a \in \Q$ with 
$(a,a,a,a,a, \dots, a, \dots )\in \A.$

The group $\A^\times$ of units of $\A$ is called the ideles. 
It is described explicitly as 
$$
\{ x \in \R^\times \times \prod_p \Q_p^\times: \text{ the set of primes } p \text{ such that } x_p \notin \Z_p^\times \text{ is finite}\}.
$$
Notice that if $a \in \Q^\times$ then $(a,a,a,a,a, \dots, a, \dots )\in \A^\times,$ since $a \in \Z_p^\times$
for any prime $p$ which divides neither
the numerator nor the denominator of $a$ (in lowest terms). 

Like $\Q_p,$ the ring $\A$ possesses a measure, called Haar measure, 
which is unique up to multiplication by 
a positive scalar,  assigns open sets positive measure, 
assigns compact sets finite measure, and is invariant under additive 
translation. It can be pinned down uniquely by requiring that on each 
of the subsets 
$$
\R \times \prod_{p< T}\Q_p \times \prod_{p \ge T}\Z_p,  \qquad (T \in (0,\infty) ),  
$$
it is given by the product of Lebesgue measure on $\R$ and Haar measure
on $\Q_p$ or $\Z_p.$ (Notice that $\A$ is the union of these 
subsets.)

We shall also need to use the idelic absolute value, which is 
defined as follows. Suppose that $x \in \A^\times.$
Let $S = \{ p\text{ prime}\mid x_p \notin \Z_p^\times \}.$ 
Then we define $|x| = |x_\infty|_\infty \cdot \prod_{p \in S} |x_p|_p.$
Here $|x_\infty|_\infty$ is just the usual absolute value of $x_\infty,$
which is an element of $\R^\times.$
Notice that if $S'$ is any finite set which contains $S$ then
$|x| = |x_\infty| \cdot \prod_{p \in S'} |x_p|_p$ (since we've only 
added a finite number of $1$'s to the product). Indeed, the infinite
product $|x_\infty| \cdot \prod_{\text{all primes }p } |x_p|_p$
converges to $|x|$ because all but finitely many terms in the product 
are $1.$ The idelic absolute value has the key properties that if $x,y \in \A^\times$ then $|xy| = |x||y|,$ and that if $a \in \Q^\times$
then $|a| = 1.$

Finally, we would like to define a continuous homomorphism 
$\A \to S^1$ which is trivial on the embedded copy of $\Q.$
Previously we introduced $e_p : \Q_p \to S^1.$ Take 
$x = (x_\infty, x_2, x_3, \dots, x_p, \dots) \in \A.$
Then, for all but a finite number of primes $p,$ 
$x_p\in \Z_p$ and hence $e_p(x_p) = 1.$ Hence 
\begin{equation}\label{eq:e(x)}
e(x) = e^{2\pi i x_\infty} \prod_p e_p(x_p)\end{equation}
does not have any convergence issues, and gives a well defined 
continuous homomorphism $\A \to S^1.$ It can be shown that 
it is trivial on $\Q.$

\subsection{The group \texorpdfstring{$G_2$}{G₂}} 
In this section we briefly introduce a certain matrix group 
which we call ``$G_2.$'' Some references for this 
material are 
\cite{FultonHarris},\cite{borel-LAG},\cite{humphreys-LAG},
\cite{springer-LAG},\cite{Weissman}.

The name ``$G_2$'' applies, first and foremost, to a root system. 
A root system is a finite spanning set $\Phi$ in a real inner product space (one may take it to be $\R^n$ with the usual dot product) such that
\begin{enumerate}
    \item the reflection $\alpha - 2 \frac{\alpha \cdot \beta}{\beta \cdot \beta} \beta$ of $\alpha$ in the plane orthogonal to $\beta$
    is an element of $\Phi$ for all $\alpha, \beta 
    \in \Phi,$
    \item the quantity $2 \frac{\alpha \cdot \beta}{\beta \cdot \beta}$
    is an integer for all $\alpha, \beta \in \Phi,$
    \item if $\alpha \in \Phi, k \in \R$ and $k \alpha \in \Phi,$ then $k=1$ or $-1.$
\end{enumerate}
The elements of the set $\Phi$ are called roots.
The second requirement has an interesting geometric consequence. 
Indeed, if $\alpha$ and $\beta$ are two roots,
and if $\theta$ is the angle between them then 
$$\cos^2\theta = \frac{\alpha \cdot \beta}{\beta \cdot \beta}
\frac{\alpha \cdot \beta}{\alpha \cdot \alpha}
\in \{ 0, \frac 14, \frac 12, \frac 34 , 1\}.
$$
Up to rotation and scaling, there are only four root systems in $\R^2.$
They are shown below.
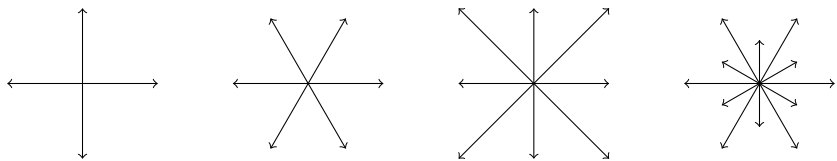
\begin{center}
\begin{tikzpicture}
\draw[<->] (-1,0)--(1,0);
\draw[<->] (0,-1)--(0,1);
\draw[<->] (2,0)--(4,0);
\draw[<->] (2.5,0.866)--(3.5,-0.866);
\draw[<->] (2.5,-0.866)--(3.5,0.866);
\draw[<->] (5,0)--(7,0);
\draw[<->] (6,-1)--(6,1);
\draw[<->] (5,-1)--(7,1);
\draw[<->] (7,-1)--(5,1);
\draw[<->] (8,0)--(10,0);
\draw[<->] (8.5,0.866)--(9.5,-0.866);
\draw[<->] (8.5,-0.866)--(9.5,0.866);
\draw[<->] (8.5,0.288675)--(9.5,-0.288675);
\draw[<->] (8.5,-0.288675)--(9.5,0.288675);
\draw[<->] (9,-0.57735)--(9,0.57735);
\end{tikzpicture}
\end{center}
The root system at the far right is $G_2.$ It has 12 roots.

Root systems play an important role in the classification theory 
of various mathematical objects, including Lie groups, Lie 
algebras, and linear algebraic groups. For us, the important 
object is a linear algebraic group. 

A linear algebraic group $G$ over a field $F$ is defined by a positive integer $n$
and a finite set $S$ of polynomial equations, with coefficients in $F$ where the variables are the entries of an $n\times n$
matrix, and  
$$G(R):=\{ g \in GL_n(R): S \}$$
is closed under matrix multiplication and inversion, 
i.e., a subgroup of $GL_n(R)$, for any commutative $F$-algebra
with unity $R.$ For example, if $X$ is a fixed $n\times n$
matrix with entries in $F$ then 
$$\{ g \in GL_n(R): gXg^t = X\}$$
(here $^t$ denotes transpose)
is well-defined a subgroup of $GL_n(R)$ for 
any commutative $F$-algebra with unity $R.$ Thus $G$ 
itself is not actually a group, but a mapping from 
$F$-algebras to groups. 

In order to describe the results of Jiang, Rallis, and Pleso, we 
will not need to worry about the specific set of polynomial 
equations that defines $G_2.$ The key point about our ``$G_2$'' 
being a linear algebraic group is that it is not a single group, 
but a mapping from $\Q$-algebras to groups, which we will use to define 
various groups: $G_2(\Q),\ G_2(\Q_p), \ G_2(\A),$ etc., and each 
subgroup of it which we introduce is a similar object.

\subsection{Some Key Subgroups of \texorpdfstring{$G_2$}{G₂}}

Without getting too bogged down in the details of how root systems
are used in the classification of linear algebraic groups, let us 
briefly mention that our linear algebraic group $G_2$ has a one 
dimensional subgroup attached to each of the roots in the root 
system $G_2.$ For example, if we use the same embedding of $G_2$ into  $GL_8$ as Pleso \cite{pleso2009integrals},  then one of the six long roots, which we denote $\alpha,$ is attached to the group of all matrices of the form 
$$
x_\alpha(a) = \begin{pmatrix}1 & 0 & 0 & 0 & 0 & 0 & 0 & 0 \\
0 & 1 & a & 0 & 0 & 0 & 0 & 0 \\
0 & 0 & 1 & 0 & 0 & 0 & 0 & 0 \\
0 & 0 & 0 & 1 & 0 & 0 & 0 & 0 \\
0 & 0 & 0 & 0 & 1 & 0 & 0 & 0 \\
0 & 0 & 0 & 0 & 0 & 1 & -a & 0 \\
0 & 0 & 0 & 0 & 0 & 0 & 1 & 0 \\
0 & 0 & 0 & 0 & 0 & 0 & 0 & 1
\end{pmatrix},
$$
and one of the six short roots, which we denote $\beta$,
is attached to the group of all matrices of the form 
$$x_\beta(b) = \begin{pmatrix}1 & b & 0 & 0 & 0 & 0 & 0 & 0 \\
0 & 1 & 0 & 0 & 0 & 0 & 0 & 0 \\
0 & 0 & 1 & b & b & -b^{2} & 0 & 0 \\
0 & 0 & 0 & 1 & 0 & -b & 0 & 0 \\
0 & 0 & 0 & 0 & 1 & -b & 0 & 0 \\
0 & 0 & 0 & 0 & 0 & 1 & 0 & 0 \\
0 & 0 & 0 & 0 & 0 & 0 & 1 & -b \\
0 & 0 & 0 & 0 & 0 & 0 & 0 & 1\end{pmatrix}.$$

We may assume that the root system is embeeded in 
the plane so that $\alpha$ and $\beta$ are as displayed below 
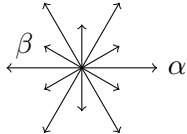
\begin{center}
\begin{tikzpicture}
 \draw[<->] (8,0)--(10,0) node[right = 0.1mm] {$\alpha$};
 \draw[<->] (8.5,0.866)--(9.5,-0.866);
 \draw[<->] (8.5,-0.866)--(9.5,0.866);
 \draw[<->] (8.5,0.288675) node[left = 0.1mm ]{$\beta$}--(9.5,-0.288675);
 \draw[<->] (8.5,-0.288675)--(9.5,0.288675);
 \draw[<->] (9,-0.57735)--(9,0.57735);
\end{tikzpicture}
\end{center}
Then the other ten roots are 
$\alpha+\beta, \alpha+2\beta, \alpha+3\beta, 2 \alpha+3\beta,$
and the negatives of the first six.

Bundling together five of these groups, we obtain another important 
group, which we denote $N$ consisting of all matrices of the form 
$$
\begin{aligned}
n(x,y,z,u,v)&= 
x_{\alpha+3\beta}(x) x_{\alpha+2\beta}(y) x_{2\alpha+3\beta}(z) x_{\alpha + \beta}(u) x_\alpha(v)\\&=
\begin{pmatrix}
    1 & 0 & -u & -y & -y & -x & v x + 2 \, u y - z & -u x - y^{2} \\
0 & 1 & v & u & u & -y & -u^{2} + v y & -u y + z \\
0 & 0 & 1 & 0 & 0 & 0 & y & x \\
0 & 0 & 0 & 1 & 0 & 0 & -u & y \\
0 & 0 & 0 & 0 & 1 & 0 & -u & y \\
0 & 0 & 0 & 0 & 0 & 1 & -v & u \\
0 & 0 & 0 & 0 & 0 & 0 & 1 & 0 \\
0 & 0 & 0 & 0 & 0 & 0 & 0 & 1
\end{pmatrix}
\end{aligned}
$$
Note that the function $n$ is not a homomorphism, because the group $N$ is not abelian. However it does give a bijection 
$R^5 \to N(R)$ for any $\Q$-algebra $R.$ If $R = \R, \Q_p$
or $\A,$ then this bijection with $R^5$ defines a 
measure on $N(R)$, which corresponds to product measure on 
$R^5.$ 

Abusing notation, we shall use the symbol $n$ for a general 
element of $N$ as well as for the above function 
from $R^5 \to N(R)$ for each $R.$ So, for example,
``$n \in N(\A)$'' means the same things as  ``$n=n(x,y,z,u,v)$ for some 
$x,y,z,u,v \in \A.$''

The function $n$ also has a lower-triangular analogue,
which will also arise in the results of Jiang-Rallis and Pleso. 
It is given by 
$$
\begin{aligned}
n^-(x,y,z,u,v) &= 
x_{-\alpha}(x) x_{\alpha-\beta}(y) x_{-2\alpha-3\beta}(z) 
x_{-\alpha-2\beta}(u) x_{-\alpha-3\beta}(v)
\\&=
\begin{pmatrix}
1 & 0 & 0 & 0 & 0 & 0 & 0 & 0 \\
0 & 1 & 0 & 0 & 0 & 0 & 0 & 0 \\
-y & x & 1 & 0 & 0 & 0 & 0 & 0 \\
-u & y & 0 & 1 & 0 & 0 & 0 & 0 \\
-u & y & 0 & 0 & 1 & 0 & 0 & 0 \\
-v & -u & 0 & 0 & 0 & 1 & 0 & 0 \\
v x + 2 \, u y - z & u x - y^{2} & u & -y & -y & -x & 1 & 0 \\
-u^{2} - v y & -u y + z & v & u & u & y & 0 & 1
\end{pmatrix},    
\end{aligned}
$$
and its image is a subgroup of $G_2$ which is conjugate to $N$ and denoted $N^-.$ As we did with ``$n$'' we will also 
abuse notation with $n^-$ by denoting a general element of 
the group $N^-(R)$ by $n^-$ when we have no need to 
refer to its individual coordinates.

The next subgroup of $G_2$ is isomorphic to $GL_2.$ Specifically, 
if $g = \left(\begin{smallmatrix} a&b\\c&d \end{smallmatrix} \right),$ then we define
$$
m(g)= \begin{pmatrix}
    a&b&&&&&&\\ c&d&&&&&&\\ 
    && a^2/\Delta&ab/\Delta&ab/\Delta&-b^2/\Delta&&\\
    && ad/\Delta&ad/\Delta&bc/\Delta&-bd/\Delta&&\\
    && ad/\Delta&bc/\Delta&ad/\Delta&-bd/\Delta&&\\
    && -d^2/\Delta&-cd/\Delta&-cd/\Delta&d^2/\Delta&&\\
    &&&&&& a/\Delta&-b/\Delta\\
    &&&&&& -c/\Delta&d/\Delta 
\end{pmatrix},
$$
where $\Delta = ad-bc.$
Then $m:GL_2 \to GL_8$ is a 
homomorphism and its image is 
a subgroup of $G_2$ which we denote $M.$

The group $M$ normalizes the group $N,$ and the product of the groups $N$ and $M$ is denoted by $P.$
For any $\Q$-algebra $R$, every element of 
$P(R)$ can be written uniquely as 
$m(g)n(x,y,z,u,v)$ for some $g \in GL_2(R)$ and 
$x,y,z,u,v\in R.$ Since $M$ normalizes $U,$ the 
mapping $m(g)n(x,y,z,u,v) \to g$ is a homomorphism.

For any prime $p$ we have $G_2(\Q_p) = P(\Q_p) \cdot G_2(\Z_p),$
that is, every element of $G_2(\Q_p)$ can be expressed 
as $b k$ with $b \in P(\Q_p)$ and $k \in G_2(\Z_p).$
References for this 
fact include \cite{Iwahori-Matsumoto} Proposition 2.33.
The  expression is not unique, because $P(\Q_p) \cap G_2(\Z_p) 
= P(\Z_p) $ which is not empty. But if $b_1k_1=b_2k_2,$ then 
there exists $\beta \in P(\Z_p)$ such that $b_2 = b_1 \beta$
and $k_2 = \beta^{-1} k_1.$

\subsection{Characters of \texorpdfstring{$N$}{N} and their connection with cubic forms}

We shall be working with continuous homomorphisms from $N(\A)$ and $N(\Q_p)$ 
into the group $S^1= \{ z \in \C: |z|=1\}.$ Such homomorphisms are 
sometimes called ``characters.''

Let $\nu(n(x,y,z,u,v))$ be the row
vector $\bbm x&y&u&v\ebm.$
It's not difficult to 
check that $\nu$ 
is a homomorphism $N(R) \to R^4$ for any $\Q$-algebra $R$. Its kernel is 
the group $Z(R)$ of all matrices of the form $n(0,0,z,0,0).$
It's also not hard to check that 
$$n(0,0,0,0,z)n(1,0,0,0,0)n(0,0,0,0,z)^{-1}n(1,0,0,0,0)^{-1}
= n(0,0,z,0,0).$$ This implies that, for any $\Q$-algebra $R$, 
and any abelian group $A,$ every homomorphism from $N(R)$ to $A$ 
must factor through $\nu.$

Next, we need some key facts about the space of all 
continuous homomorphisms $R \to S^1$ when $R = \R, \Q_p$ 
or $\A.$ See, for example \cite{CassFroh},\cite{Lang}.
Let $\psi$ be any fixed nontrivial continuous 
homomorphism $\A^\times \to S^1,$ which is trivial on $\Q^\times.$
(For example, the function $e$ defined in 
\eqref{eq:e(x)}.)
Then
every other continuous 
homomorphism $\Q^\times \backslash \A^\times \to S^1,$ is of the form $\psi_a(x) = \psi(ax)$
for some $a \in \Q.$ 
By restricting $\psi$ to the copy of $\R^\times$ inside 
$\A^\times$ we get a nontrivial continuous homomorphism $\psi_{\infty}:\R^\times \to S^1,$ and 
every other continuous 
homomorphism $\R^\times \to S^1,$ is of the form $\psi_{\infty, a}(x) = \psi_\infty(ax)$
for some $a \in \R.$ 
Similarly, by restricting $\psi$ to the copy of $\Q_p^\times$ inside 
$\A^\times$ we get a nontrivial continuous homomorphism $\psi_{p}:\Q_p^\times \to S^1,$ and 
every other continuous 
homomorphism $\Q_p^\times \to S^1,$ is of the form $\psi_{p, a}(x) = \psi_p(ax)$
for some $a \in \Q_p.$ 

It follows that every continuous homomorphism $ N(\A) \to S^1$ 
which is trivial on $N(\Q)$
has
the form $n(x,y,z,u,v)\mapsto \psi( a_1x+a_2y+a_3u+a_4v)$ for some 
$a_1, a_2, a_3, a_4 \in \Q,$ or, equivalently, 
$n \mapsto \psi( \nu(n) \cdot \sigma)$
for some  $\sigma\in \Q^4.$
Here $\cdot$ is the usual dot product.
Moreover,
every continuous homomorphism $N(\Q_p) \to S^1$
has the same form, with coefficients in $\Q_p.$
In either case, we let $\psi_{\sigma}$ be the mapping 
attached to the quadruple $\sigma$ as above.

We have 
\begin{equation}\label{eq:nu(m inverse n m)}  
\nu (m(g^{-1}) n(x,y,z,u,v) m(g)  ) = \bbm x&y&u&v\ebm \cdot \rho(g), 
\end{equation}
where 
$$
\rho\left( \begin{pmatrix}
    1&a\\ 0&1
\end{pmatrix}\right) = 
\begin{pmatrix}
1&0&0&0\\-3a&1&0&0\\-3a^2&2a&1&0\\-a^3&a^2&a&1
\end{pmatrix},
$$
$$
\rho\left( \begin{pmatrix}
    1&0\\a&1
\end{pmatrix}\right) = 
\begin{pmatrix}
      1&-a&-a^2&-a^3\\ 
    0&1&2a&3a^2\\
    0&0&1&3a\\
    0&0&0&1  
\end{pmatrix},
$$
$$
\rho\left( \begin{pmatrix}
    t_1&0 \\ 0 & t_1
\end{pmatrix}\right) = 
\begin{pmatrix}
    t_2/t_1^2 &0&0&0\\0&t_1^{-1} &0&0 \\ 0&0&t_2^{-1}&0 \\ 0&0&0&t_1/t_2^2
\end{pmatrix}.
$$
From \eqref{eq:nu(m inverse n m)}, it follows that 
for $\sigma \in \Q^4$ and $g \in GL_2(\Q),$
$$\psi_\sigma(m(g)^{-1} n m(g) ) = \psi_{ \sigma\cdot \rho(g)^t}(n).$$

The mapping $\rho: GL_2 \to GL_4$ is closely related to a certain 
action on homogeneous polynomials of degree at most $3$ which we now explain. Key results are from \cite{Wright}.
If $\underline{c}=(c_1, c_2, c_3, c_4),$ 
 then we define 
$$F_{\underline{c}}([x,y])= c_1x^3+c_2x^2y+c_3xy^2+c_4y^3.$$ 
For $g \in GL_2$ we define 
$$ g \cdot F([x,y]) = (\det g)^{-1}F([x,y] g).$$
Then we define $\varrho(g)$ to be the matrix such that 
$$g\cdot F_{\underline{c}} = F_{\underline{c}\cdot \varrho(g)^t}.$$
It's easy to check that
$$\begin{aligned}
\varrho\left( 
\bbm t_1& 0\\ 0 & t_2\ebm 
\right)
&= \bbm t_1^2/t_2&&&\\&t_1&&\\&&t_2&\\ &&&t_2^2/t_1 \ebm\\
\varrho\left( 
\bbm 1& a\\ 0 & 1\ebm 
\right)
&= \bbm 1&a&a^2&a^3\\&1&2a&3a^2\\&&1&3a\\ &&&1 \ebm\\
\varrho\left( 
\bbm 1& 0\\ a & 1\ebm 
\right)
&= \bbm 1&&&\\3a&1&&\\3a^2&2a&1&\\ a^3&a^2&a&1 \ebm
\end{aligned}
$$

Then 
\begin{equation}
\label{rhoAndvarrho}
\rho(g) = 
w_1
\varrho(\det g^{-1} g)
w_1^{-1},
\qquad \text{ where } w_1 = \begin{pmatrix}  
0&0&0&-1\\  0&0&1&0\\ 0&1&0&0\\1&0&0&0 
\end{pmatrix}.
\end{equation}
Indeed, it suffices to check this identity on matrices of the form 
$\bpm 1&a\\ 0&1 \epm,$
$\bpm 1&0\\ a&1 \epm,$ and 
$\bpm t_1&0\\ 0&t_2 \epm,$
because matrices of these three forms generate $GL_2.$ 

\subsection{Some results of Wright}

Let $k$ be an arbitrary field. Orbits for the action of $GL_2(k)$ on $k^4$ by $\varrho^t$
were classified in \cite{Wright}.
\begin{thm}[\cite{Wright}] 
    The orbits for the action of $GL_2(k)$ on $k^4$ by $\varrho^t$
    are as follows: 
    \begin{enumerate}
        \item The set which contains only the zero vector is a singleton orbit.
        \item The set which consists of all $\underline{c}$ 
        such that $F_{\underline{c}}(x,y) = d(ax+by)^3$ for 
        some (non-unique) $a,b,d \in k$ is a single orbit.
        \item The set which consists of all $\underline{c}$ 
        such that $F_{\underline{c}}(x,y) = (a_1x+b_2y)^2(a_2x+b_2y)$ for 
        some (non-unique) $a_1,a_2,b_1,b_2 \in k$ is a single orbit.
        \item For all other $\underline{c},$ let $K_{\underline{c}}$ be the smallest field extension of 
        $k$ such that $F_{\underline{c}}$ splits into linear 
        factors over $K_{\underline{c}}.$ Then $\underline{c}$
        and $\underline{d} \in k^4$ lie in the same $GL_2(k)$-orbit if and only if $K_{\underline{c}}=K_{\underline{d}}.$
    \end{enumerate}
\end{thm}

Let $f_{\underline{c}}(x) = F_{\underline{c}}(1,x).$ We give a 
modest reformulation of the result. 
\begin{cor} Assume that $k$ has more than two elements.
    Each nonzero orbit for the action of $GL_2(k)$ on $k^4$ by 
    $\varrho$ contains an element such that $f_{\underline{c}}$
    has degree $3.$ To such an element $\underline{c}$ we can attach the quotient ring $k[x]/(f_{\underline{c}}).$ Then 
    two elements of $k^4$ are in the same orbit if and only 
    if the attached quotient rings are isomorphic.
\end{cor}
\begin{proof}
    Indeed, $f_{\underline{c}}$ has degree less than $3$ if and only if $x$ is a linear factor of $F_{\underline{c}}.$ By 
    Wright's result, the orbit of $\underline{c}$ determines the 
    number and multiplicities of the linear factors of $F_{\underline{c}},$ but not what they are. 
    So its possible to choose an element such that none of the 
    linear factors is $x,$ with only one exception-- if $k$ has two elements and $F_{\underline{c}}$ has three distinct linear factors, then one of them must be $x$ because $k[x,y]$ only has three linear elements.
    
    If $f_{\underline{c}}$ has distinct irreducible factors
    then, by the Chinese remainder theorem, $k[x]/(f_{\underline{c}}(x))$ is the direct sum of the quotients of $k[x]$ by the individual factors-- that is, it's a single cubic field if $f_{\underline{c}}$ is irreducible, or $k \oplus K_{\underline{c}}$ if $f_{\underline{c}}$ is the
    product of a linear factor and a quadratic factor, and it's 
    $k \oplus k \oplus k$ if $f_{\underline{c}}$ is the
    product of three distinct linear factors. Finally, if 
    $f_{\underline{c}}$ has a repeated root $r$ then, substituting $x \to x-r$ (which is an automorphism of $k[x]$), we can assume $r=0.$ Then $k[x] \cong k\oplus k[x]/(x^2)$
    if $0$ is  a double root, or $k[x]/(x^3)$ if it's a triple root. 
    No two of the rings thus obtained are isomorphic to each other.
\end{proof}
\begin{rmks}
\begin{enumerate}
    \item If we had defined $f_{\underline{c}}(x)$ as $F_{\underline{c}}(x,1)$ we would have gotten the same 
    quotient ring for each orbit. Indeed, each orbit has an 
    element which is divisible by neither $x$ nor $y.$ For such 
    an element $F_{\underline{c}}(x,1)$ and $F_{\underline{c}}(1,x)$ are both cubic, and their roots are the inverses of 
    one another.
\item    If $k[x]/(f_{\underline{c}}(x))$ is a cubic field, there are two possibilities for its relationship with the field $K_{\underline{c}}$ introduced earlier. Let $\alpha$ be any 
    root of $f_{\underline{c}}(x).$ Then $k[x]/(f_{\underline{c}}(x))\cong k(\alpha).$ If $k(\alpha)$ contains the other two roots of $f_{\underline{c}}(x),$ then $K_{\underline{c}} = k(\alpha) \cong k[x]/(f_{\underline{c}}(x)).$ If not, then the two other roots of $f_{\underline{c}}$ lie in a quadratic 
    extension of $k(\alpha),$ and this extension is $K_{\underline{c}}.$ It is, of course, degree $6$ over $k.$
\end{enumerate}
\end{rmks}
This also enables us to compute the orbits for the 
action of $GL_2(k)$ on $k^4$ via $\rho.$
Because of equation \eqref{rhoAndvarrho}, which 
relates the actions by $\rho$ and 
by $\varrho,$ it is convenient to 
make the following definition.  
    For $\sigma=(a_1, a_2, a_3, a_4)
    \in k^4$ let
    $$g_\sigma = f_{\sigma \cdot w_1} =f_{(a_4, a_3, a_2, -a_1)} \in k[x].$$
\begin{cor}
Let $k$ be any field.
    Take $\sigma$ and $\tau \in k^4.$  
    Then there exists $g \in GL_2(k)$ such that 
    $\tau= \sigma \rho(g)^t$ if and only if 
    $$k[x]/(g_{\sigma})\cong  k[x]/(g_{\tau}).$$
\end{cor}
In what follows, 
it will also be useful 
to have a partial description of the orbits of
$GL_2(\Z_p)$ acting on $\Q_p^4$ via $\varrho.$
Before proceeding with this it will be useful to 
introduce a key fact from \cite{Wright}.
For $k$ any field and 
$\underline{c} = (c_1, c_2, c_3, c_4) \in k^4$
let 
$$P(\underline{c})=
c_2^2c_3^2+18c_1c_2c_3c_4-4c_2^3c_4-4c_1c_3^3-27c_1^2c_4^2.
$$
\begin{lem}\label{P(c.g)=det^2gP(c)}
For all fields $k,$ all $\uline{c}\in k^4$ and all $g \in GL(2, k),$ we have $P(\uline{c} \cdot \varrho(g)^t) = 
\det g^2 P(\uline{c}).$    
\end{lem}
\begin{proof}
    This appears near the top of p. 512 of \cite{Wright}.
\end{proof}
If $f_{\underline{c}}$ is a cubic polynomial, with roots
(possibly lying in an extension field)
$\alpha_1, \alpha_2$ and $\alpha_3,$ then 
$P(\uline{c})$ is known as its discriminant, and is also given by 
the formula
$$
P(\underline{c})=c_4^4\prod_{1\le i<j \le 3} (\alpha_i-\alpha_j)^2.$$
In particular, it is zero if and only if $f_{\underline{c}}$ 
has a repeated root.
See, for example, \cite{DummitFoote}, pp. 610-612.

\begin{prop}
    Take $\underline{c}$ and $\underline{d} \in \Q_p^4.$ 
    If $\underline{c}$ and $\underline{d}$ are in the same 
    $GL_2(\Z_p)$-orbit, then $K_{\underline{c}}=K_{\underline{d}},$
    $|P(\uline{c})|_p = |P(\uline{d})|_p,$ and
    $\max(|c_1|_p, |c_2|_p, |c_3|_p, |c_4|_p)
    = \max(|d_1|_p, |d_2|_p, |d_3|_p, |d_4|_p).$
    Moreover, for each relevant field $K$ 
    $$\{\uline{c}\in \Z_p^4: |P(\uline{c} )|_p=1, \text{ and } K_{\underline{c}}=K\}$$
    is a single $GL_2(\Z_p)$-orbit.
\end{prop}
\begin{rmk}
    For $\uline{c} \in \Q_p^4$ such that $|P(\uline{c})|=1,$
    $\uline{c} \in \Z_p^4 \iff \max(|c_1|_p, |c_2|_p, |c_3|_p, |c_4|_p)=1.$
\end{rmk}
\begin{proof}
    First assume that $\underline{c}$ and $\underline{d}$ are in the same 
    $GL_2(\Z_p)$-orbit. Then $K_{\underline{c}}=K_{\underline{d}},$ because they 
    are in the same $GL_2(\Q_p)$ orbit, and $|P(\uline{c})|_p = |P(\uline{d})|_p,$ by lemma \ref{P(c.g)=det^2gP(c)}
    (since the determinant of an element of $GL_2(\Z_p)$ lies
    in $\Z_p^\times$). 

    Now, $\max(|c_1|_p, |c_2|_p, |c_3|_p, |c_4|_p)=p^{-k}$ if and 
    only if $k$ is the smallest integer such that 
    $\uline{c} \in p^k \cdot \Z_p^4.$ From this description, it's 
    clear that $\max(|d_1|_p, |d_2|_p, |d_3|_p, |d_4|_p)
    \le  \max(|c_1|_p, |c_2|_p, |c_3|_p, |c_4|_p)$ whenever 
    $\uline{d}  = \uline{c} h$ for some $4\times 4$ matrix 
    $h$ with entries in $\Z_p.$ In particular, equality 
    holds when $h \in GL_4(\Z_p),$ and this includes the case 
    $h = \rho(g)^t$ for some $g \in GL_2(\Z_p).$

    Now, take $\uline{c},\uline{d} \in \Z_p^4$ such that 
    $|P(\uline{c} )|_p = |P(\uline{d})|_p = 1$ and 
    $K_{\uline{c}} = K_{\uline{d}}.$ 
    We must prove that $\ul c$ and 
    $\ul d$ are in the same orbit. 
    There are three cases, depending 
    on whether $f_{\ul c}$ is irreducible, the product of a linear factor by an irreducible quadratic, or the 
    product of three linear factors. The case when 
    $f_{\ul c}$ is irreducible is most challenging, and 
    we tackle it in several steps. 

    The first step is to show that  
     the element $\overline{f}_{\uline{c}}$ of $\Z/p\Z[x]$ 
    obtained by reducing the coefficients of $f_{\uline{c}}$ 
    mod $p$ is also irreducible.   
    Since $P(\uline{c}) \not \equiv 0 \pmod{p},$ it
    follows that $\ol f_{\ul c}$
    can't have a double zero. If it had a simple zero, 
    then, by Hensel's lemma, $f_{\uline{c}}$ 
    would have a zero in $\Z_p.$ Thus $\overline{f}_{\uline{c}}$
    must be irreducible. This implies that $c_1$ and $c_4$ are
    in $\Z_p^\times.$ 
    
    Since $K_{\uline d} = K_{\uline c},$ it follows that
    $f_{\uline d}$ must also be irreducible and remain irreducible mod $p,$ and hence that $d_1$ and $d_4$ must also be units. Acting by a scalar matrix in $GL_2(\Z_p)$ we can multiply all entries of $\uline c$ (or $\uline d$) by any scalar in $\Z_p^\times.$ From this, it follows that 
    it suffices to treat the case when $f_{\uline{c}}$ and $f_{\uline{d}}$ are monic.

For the next step, we consider the
finite extension $\Q_p(x)/(f_{\uline{c}}(x)).$
It must contain a root of $f_{\uline{d}}.$
Let $\beta= b_0 + b_1x+b_2x^2 + (f_{\uline{c}})$ 
be some such root. A priori,  $b_0,b_1, b_2$
are in $\Q_p.$ Our next step is to prove that they are
actually in $\Z_p.$ And that at least one of $b_1, b_2$
is in $\Z_p^\times.$

Suppose not. Then there is a unique positive integer
    $k$ such that $p^kb_0, p^kb_1,$ and $p^k b_2$ are all 
    in $\Z_p$ and at least one of them is in $\Z_p^\times.$
    Reducing mod $p$ gives a nonzero element of the finite 
    field $\Z/p\Z(x)/(\overline{f}_{\uline{c}}).$
    
    But $\beta$ was a zero of $f_{\uline{d}}(x) = x^3+d_3x^2+d_2x+d_1,$ so $p^k\beta$ is a zero of 
    $x^3+d_3p^kx^2+d_2p^{2k}x+d_1p^{3k}.$ So, its image in 
    $\Z/p\Z[x]/(\overline{f}_{\uline{c}})$ must be a zero 
    of $x^3$ and this is a contradiction. Thus $b_0,b_1, b_2 \in \Z_p.$

    If we reduce $b_0,b_1, b_2$ mod $p$ we get a zero 
    of $\overline{f}_{\uline{d}}$ in 
     $\Z/p\Z[x]/(\overline{f}_{\uline{c}}).$ Since  $\overline{f}_{\uline{d}}$ is irreducible, it can't be 
     in $\Z/p\Z,$ so at least one of $b_1, b_2$ must be in 
     $\Z_p^\times.$

     Next, we prove that there exist $a_0,a_1, e_0, e_1\in \Z_p$
     such that 
     $(a_0+a_1x) \cdot \beta \equiv (e_0+e_1x)\pmod{f_{\uline{c}}(x)}.$
     Indeed, $x^3 \equiv -(c_1+c_2x+c_3x^2)\pmod{f_{\uline{c}}(x)},$
     so
     for $a_0,a_1 \in \Z_p$ we have
     $$(a_0 + a_1x)\beta = a_0b_0-a_1b_2c_1 + (a_1b_0 +b_1a_0-a_1b_2c_2)x + (a_0b_2+a_1b_1-a_1b_2c_3)x^3 
     \pmod{f_{\uline{c}}(x)}.
     $$
     If 
     $b_2$ is a unit then we can take $a_1=1, a_0 = c_3-b_1b_2^{-1}.$ And if $b_2 \in p\Z_p$ then $b_1 \in \Z_p^{\times}.$ But then $b_1-b_2c_3$ is also in $\Z_p^\times,$ and we can take $a_0=1$ and $a_1 = -b_2(b_1-b_2c_3)^{-1}.$

    Now we are ready to prove that $\uline{c}$ and $\uline{d}$ are in the same $GL_2(\Z_p)$-orbit.
     Let $\alpha, \varepsilon,$ and $\gamma$ be the images
    of $a_0+a_1x, e_0+e_1x$ and $x$ in the quotient field 
    $\Q_p[x]/f_{\uline{c}}(x)).$ Note that $(\alpha, \varepsilon)=(1, \gamma)g$ where $g = \left(\begin{smallmatrix}
        a_0&e_0\\ a_1& e_1 
    \end{smallmatrix}\right).$ 
    Also, $\beta = \alpha^{-1} \epsilon.$
    The matrix $g$ has entries in $\Z_p.$ We claim that it is 
    in $GL_2(\Z_p).$ If it were not, its columns would be linearly dependent modulo, $p.$ But if this were the case, then $\alpha^{-1} \epsilon$ would be an element of $\Z_p^\times,$ and it's not.
    Since $f_{\overline{d}}$ vanishes
    at $\beta,$ it follows that $F_{\overline{d}}$ vanishes
    at $(\alpha, \varepsilon)$. But then 
    $F_{\overline{d} \cdot \varrho(g)^t}$ vanishes at $(1, \gamma),$ and $f_{\overline{d} \cdot \varrho(g)^t}$ vanishes at $\gamma.$ This forces $\overline{d} \cdot \varrho(g)^t$ 
    to be a scalar multiple of $\uline{c},$ so $\uline{c}$ and $\uline{d}$ are in the same orbit.

    We are done in the case when $f_{\uline{c}}$ is irreducible.
    We now consider the case when it is reducible. 

    The first step is to prove that, whenever 
    $f_{\uline{c}}$ is reducible, $\uline{c}$ is in the 
    same orbit as an element of the form $(c_1', c_2' 1, 0).$  Assume $ \min(|c_1|_p, |c_2|_p, |c_3|_p, |c_4|_p) =|P(\uline{c} )|_p=1,$ and $F_{\uline{c}}$ factors 
    as $(a_1x+a_2y)(b_1x^2+b_2xy+b_3y^2)$ with all coefficients
    in $\Q_p.$ Then it's not hard to show that it has a 
    factorization in the same form where all coefficients are
    in $\Z_p$ and there exists $i,j$ such that $a_i$ and $b_j$
    are in $\Z_p^\times.$ But then $(a_1, a_2)^t$ is the first column of an element $g$ of $GL_2(\Z_p).$ Acting by the 
    inverse we may transform $(a_1x+a_2y)$ to $x.$ Thus, any quadruple $\uline{c}$ such that $F_{\uline{c}}$ is reducible
    is in the same $GL_2(\Z_p)$ orbit as one where $c_4=0.$
    But if $c_4=0,$ then $P(\uline{c})=c_3^2(c_2^3-4c_1c_3).$ If 
    all entries are in $\Z_p$ and this is in $\Z_p^\times,$ then 
    $c_3 \in \Z_p^\times.$ Since each scalar matrix acts by the 
    corresponding scalar, we deduce that any quadruple $\uline{c}$ such that $F_{\uline{c}}$ is reducible is in the 
    same $GL_2(\Z_p)$ orbit as an element with $c_4=0, c_3=1.$
    So, suppose $F_{\uline{c}}(x,y) = x(c_1x^2+c_2xy+y^2).$

    If $c_1x^2+c_2xy+y^2$ factors in $\Q_p[x,y]$ then it factors
    as $(a_1x+y)(a_2x+y)$ where $a_1,a_2 \in \Z_p.$ Acting by $\left( \begin{smallmatrix}
        1&-a_1\\ 0 & 1
    \end{smallmatrix}\right)$ transforms $a_1x+y$ to $y,$ 
    while leaving $x$ fixed. Thus, 
    if $F_{\uline{c}}$ is a product of distinct linear factors, then it is in the same $GL_2(\Z_p)$-orbit as $(0,c_2',1,0)$
    where $c_2'=a_2-a_1.$ Now, 
    $P((0,c_2',1,0))=(c_2')^2,$ so $c_2' \in \Z_p^{\times}.$ 
    Then $\left( \begin{smallmatrix}
        1&0 \\ 0 & c_2'
    \end{smallmatrix}\right)$ transforms $xy(c_2'x+y)$ to 
    $xy(x+y).$ Thus, any tuple $\uline{c}$ such that $F_{\uline{c}}$ is a product of distinct linear factors
    is in the same $GL_2(\Z_p)$-orbit as $(0,1,1,0).$

    Now suppose $c_1x^2+c_2xy+y^2$ does not factor. If $p$ is not $2$ then we can act by $\left( \begin{smallmatrix}
        1&-c_2/2 \\ 0&1 
    \end{smallmatrix}\right)$ to get rid of $c_2.$ That is, 
    each $GL_2(\Z_p)$-orbit with a linear factor 
    contains an element of the form $(c_1,0,1,0).$  
    Since $P((c_1,0,1,0))=-4c_1,$ we deduce (still under the 
    assumption that $p\ne 2$) that $c_1 \in \Z_p^\times.$
    And since $$(c_1, 0,1,0) \cdot \varrho\begin{pmatrix}
     a& 0 \\ 0 & 1\end{pmatrix} = (a^2c_1,0,1,0),$$
    it follows that 
     two quadruples which generate the same quadratic extension
    are in the same $GL_2(\Z_p)$-orbit. 
    
    If $p=2$ this approach 
    doesn't work: for example, the $GL_2(\Z_p)$-orbit of $x^2+xy+y^2$ does not contain any element of the form $cx^2+y^2$ 
    with $c \in \Z_2.$ However, if $x^2+c_2x+c_1$
    and $x^2+d_2x+d_1$ generate the same quadratic extension, 
    then a simpler version of the argument given in the cubic case works. Take $\gamma$ a zero of $x^2+c_2x+c_1$ and 
    $\beta$ a zero of $x^2+d_2x+d_1$ then $\beta=b_0+b_1\gamma$
    where $b_0 \in \Z_p$ and $b_1 \in \Z_p^{\times}.$
    Then $\uline{c}$ and $\uline{d} \cdot \varrho\left( \begin{smallmatrix}
        1& b_0 \\ 0 & b_1
    \end{smallmatrix}\right)$ are unit scalar multiples of each 
    other, and hence $\uline{c}$ and $\uline{d}$ are in the same orbit.

\end{proof}

\subsection{The results of Jiang and Rallis} 

In this section we briefly describe the main results of \cite{Jiang-Rallis}. 
 First, we take a continuous function 
$f: G(\A) \times \C \to \C$
such that 
\begin{equation}\label{eq:ind rep eqvar prop}
    f_s( nm(g) h) = | \det g|^{s}f_s(h),\end{equation}
for all $h \in GL_2(\A), n \in N(\A)$ and $h \in G_2(\A).$
(Here, the value of $f$ at $h \in G_2(\A)$ and $s \in \C$
is denoted $f_s(h)$ rather than $f(h,s).$)
We also select a row vector $\sigma \in \Q^4,$
and let 
$$I^\sigma(s, f_s) := \int_{(\Q\bs \A)^5 } f_s(w_0n(x,y,z,u,v), s) \psi((v,u,y,x)\cdot \underline{\sigma}  ) \, dn,$$
and 
$$w_0 =\begin{pmatrix}
    0 & 0 & 0 & 0 & 0 & 0 & 1 & 0 \\
0 & 0 & 0 & 0 & 0 & 0 & 0 & 1 \\
0 & 0 & -1 & 0 & 0 & 0 & 0 & 0 \\
0 & 0 & 0 & 1 & 0 & 0 & 0 & 0 \\
0 & 0 & 0 & 0 & 1 & 0 & 0 & 0 \\
0 & 0 & 0 & 0 & 0 & -1 & 0 & 0 \\
1 & 0 & 0 & 0 & 0 & 0 & 0 & 0 \\
0 & 1 & 0 & 0 & 0 & 0 & 0 & 0
\end{pmatrix}.$$
Here, the integral over $(\Q/\A)$ really means an integral 
over a measurable fundamental domain, and $dn$ is the 
product measure. 
For $h_1 \in G_2(\A),$  We define 
$R(h_1)f_s(h) = f_s( hh_1).$ Notice that $R(h_1)f_s$ is another
function which satisfies \eqref{eq:ind rep eqvar prop} so 
$I^\sigma(s, R(h_1)f_s)$ is defined as well. Moreover, 
$I^\sigma( s, R(m(\gm))f_s)= I^{\sigma \cdot \rho(\gm)^t}(s, f_s),$ so it suffices to study $I^\sigma(s, f_s)$ for 
$\sigma$ ranging over a set of representatives for the 
orbits of $GL_2(\Q),$ acting on $\Q^4$ via $\rho.$

The first main result of Jiang and Rallis states that 
$$I^\sigma(s, R(g).f_s) = \int_{N(\Q) \bs N(\A)} E(ng,s; f_s) \psi_{\sigma}(n) \, dn,$$
provided that $\sigma$ corresponds to a polynomial with 
distinct roots. Here $E(g, s, ; f_s)$ is a certain 
function $G_2(\Q) \bs G_2(\A) \times \C \to \C$ called an 
Eisenstein series. We won't go into the details of this part, 
but we mention it to motivate the following. From now on 
we shall only consider $\sigma$ which are attached to 
polynomials with distinct roots.

To describe the next result, we assume not only that 
$f_s$ satisfies \eqref{eq:ind rep eqvar prop}, but also that 
$f_s$ 
factors as a special type of product. For each prime 
$p,$ we can define 
$$f_{s,p}^\circ(n m(g) k) = |\det g|_p^{s}, \qquad 
g \in GL(2, \Q_p), n\in N(\Q_p),  k \in G_2(\Z_p).$$
To see that this defines a function on all of $G_2(\Q_p)$
recall that any $h \in G_2(\Q_p)$ can be expressed as 
$bk$ with $b \in P(\Q_p)$ and $k \in G_2(\Z_p).$ Then 
$b$ can be expressed as $nm(g)$ for some $n \in N(\Q_p)$ and 
$g \in GL_2(\Q_p).$ The expression $h=bk$ is not unique, but 
if $b_1k_1=b_2k_2$ then $b_2 = b_1 \beta$ with $\beta \in P(\Z_p).$ It follows that, if $b_j=n_jm(g_j)$ for $j=1,2$
then $g_2 = g_1\gm$ for some $\gm \in GL_2(\Z_p).$ But then 
$|\det \gm|_p=1,$ so $|\det g_1|_p = |\det g_2|_p.$

We now assume that 
$$f_s(h) = f_{s,\infty}(h_\infty) \prod_p f_{s,p}(h_p)
\qquad \text{ for } h=(h_\infty, h_2, h_3, \dots, h_p, \dots) \in G_2(\A),$$
and that the set of primes such that $f_{s,p} \ne f_{s,p}^\circ$
is finite. 
(Note that this second condition ensures that the infinite
product is always convergent because all but finitely 
many of its factors are one.)
This ensures that 
$$I^\sigma( s, f_s) 
= I_\infty^\sigma  ( s, f_{s, \infty})
\cdot \prod_p I_p^\sigma( s, f_{s,p}), $$
where
$$I_p^\sigma( s, f_{s,p}) = \int_{\Q_p^5}
f_{s,p}( w_0 
n) \psi_p ( (v,u,y,x) \cdot \us ) \, dn_p,$$
and $I_\infty$ is defined analogously. Here, $dn_p$ denotes
the product measure on $\Q_p^5.$ 

The second main result of Jiang and Rallis is to compute 
$I^{\sigma}_p(s, f_{s,p}^\circ)$ under some additional 
hypotheses. 

\begin{thm}[Cf. \cite{Jiang-Rallis},theorem 2]
Assume that $\psi_p$ is trivial on $\Z_p$ but not on $p^{-1} \Z_p.$ 
    \begin{enumerate}
        \item If $\sigma = (0,1,1,0)$ then $$I^{\sigma}_p(s, f_{s,p}^\circ) = \frac{(1-p^{-3s})(1-p^{-3s+1})(1-p^{-6s+2})(1-p^{-9s+3})}{(1-p^{-3s+1})^3}.$$
        \item If $\sigma = (0,1,0,a)$ with $a \in \Z_p^\times$ such that $-a$ is not a square, then 
        $$I^{\sigma}_p(s, f_{s,p}^\circ) = \frac{(1-p^{-3s})(1-p^{-3s+1})(1-p^{-6s+2})(1-p^{-9s+3})}{(1-p^{-3s+1})(1-p^{-6s+2})}.$$
        \item If $\sigma = (1,0,0,a)$ with $a \in \Z_p^\times$ which is not a cube, then 
        $$I^{\sigma}_p(s, f_{s,p}^\circ) = \frac{(1-p^{-3s})(1-p^{-3s+1})(1-p^{-6s+2})(1-p^{-9s+3})}{(1-p^{-9s+3})}.$$
    \end{enumerate}
\end{thm}
\begin{rmks}
    \begin{enumerate}
        \item In each case the numerator is the same, and in each case, some part of it can be cancelled with the denominator-- but never the same part.
        \item The result can also be expressed in terms of 
        ``zeta functions of non-Archimedean local fields.
        \item Jiang and Rallis actually prove the corresponding result for any non-Archimedean local field, not just $\Q_p.$
        \item For $k \in GL(2, \Z_p),$ we have 
        $$I_p^{\sigma \cdot \rho(k)^t}(s, f_{s,p}^\circ)
        = I_p^\sigma( s,f_{s,p}^\circ).$$ So, the result 
        actually determines the value of $I_p^\sigma( s,f_{s,p}^\circ)$ for any $\sigma = \uline{c} w_1^{-1} \in \Z_p^4$ such that $P(\uline{c}) \ne 0$ and $K_{\uline{c}}$
        is either $\Q_p,$ or the quadratic extension obtained 
        by adjoining the square root of a unit, or the 
        cubic extension obtained by adjoining the cube 
        root of a unit.
        \item The restriction to $a \in \Z_p^\times$ is not such a big deal. For any fixed $\underline{c} \in \Q^4$ with $P(\uline{c}) \ne 0,$ we have $P(\uline{c})\in \Z_p^\times$ for all but a finite number of $p.$ If a quadruple $\uline{c}$ such that $P(\uline{c})\in \Z_p^\times$ is in the same $GL_2(\Z_p)$-orbit as an element of the form $(0,1,0,a)$ or $(1,0,0,a),$ then $a$ must be in $\Z_p^\times.$
        \item The assumption that $\psi$ is trivial on $\Z_p$ but not 
        on $\Z_p^\times$ is a natural one. If $\psi=\prod_v\psi_v$ is any 
        character, then there are only finitely many $p$ where it does
        not hold. And, since we work over $\Q$ rather than an arbitrary number field, we can take $\psi$ to be the function $e$ introduced before, and then the assumption is true for all $p.$ 
        \item The restriction to cubic extensions which 
        are obtained by adjoining the cube root of a unit is 
        more restrictive. 
        Indeed, if $p \equiv 2 \pmod{3}$ then 
        it follows from Hensel's lemma that every element of 
        $\Z_p^\times$ is a cube. 
    \end{enumerate}
\end{rmks}

In order to motivate the results of Pleso, we briefly describe 
some of the methods in Jiang-Rallis. 
First, note that 
$$
w_0 n(x,y,z,u,v)w_0^{-1} = n^{-}(-x,y,z,u,-v).
$$
It follows that 
$$\begin{aligned} I_p^{\sigma}(s, f^\circ_{s,p}) &= \int\limits_{\Q_p}\int\limits_{\Q_p}\int\limits_{\Q_p}\int\limits_{\Q_p}\int\limits_{\Q_p}f^\circ_{s,p}
(n^-(-x,y,z,u,-v) ) \psi ((v,u,y,x) \cdot \sigma) \, dx \,dy\, dz\, du \,dv,\\
& = I_++I_-,
\end{aligned}
$$
where 
$$I_+=\int\limits_{\Z_p}\int\limits_{\Q_p}\int\limits_{\Q_p}\int\limits_{\Q_p}\int\limits_{\Q_p}f^\circ_{s,p}
(n^-(-x,y,z,u,-v) ) \psi ((v,u,y,x) \cdot \sigma) \, dx \,dy\, dz\, du \,dv,$$
and $I_-$ is the similar integral where $v$ is integrated over 
$\Q_p - \Z_p$ and each of the other four variables is 
integrated over $\Q_p.$
Now, $n^-(-x,y,z,u,-v) =n^-(-x,y,z,u,0) n^-(0,0,0,0,-v),$
If $v \in \Z_p,$ then $n^-(0,0,0,0,-v) \in G_2(\Z_p),$
so  $f^\circ_{s,p}
(n^-(-x,y,z,u,-v) )= f^\circ_{s,p}
(n^-(-x,y,z,u,0) ).$ Moreover, if we assume $\sigma \in \Z_p^4,$
then $v \in \Z_p$ implies $\psi_p(\sigma_1 v) = 1.$
Hence, 
$$I_+=\int\limits_{\Q_p}\int\limits_{\Q_p}\int\limits_{\Q_p}\int\limits_{\Q_p}f^\circ_{s,p}
(n^-(-x,y,z,u,0) ) \psi ((0,u,y,x) \cdot \sigma) \, dx \,dy\, dz\, du,$$

As we mentioned before, each of the roots of the $G_2$
root system corresponds to a one dimensional subgroup of 
the group $G_2,$ and the groups $N$ and $N^-$ are each 
formed by bundling five of those groups together. For each 
root $\gamma,$ we have made a choice of parametrization
$x_\gamma$ from $R$ to the corresponding group.

Now, for each root $\gamma$ there is a homomorphism 
$\varphi_{\gamma} : SL_2 \to G_2$ such that 
$$x_\gamma(a) = \varphi_\gamma 
\begin{pmatrix}
    1&a\\ 0 &1 
\end{pmatrix}
\qquad 
x_{-\gamma}(a) = \varphi_\gamma 
\begin{pmatrix}
    1&0\\ a &1 
\end{pmatrix}.
$$
If $v \in \Q_p-\Z_p,$ then the following $SL_2$-identity
\begin{equation}\label{SL2Iwasawa}
\bpm 1& 0 \\ -v & 1 \epm 
= \bpm v^{-1}& 0\\0 & v\epm \bpm 1& -v\\ 0 & 1 \epm 
\bpm 0 & 1 \\ -1 &v^{-1}\epm, 
\end{equation}
gives rise to a $G_2$-identity
$$
x_{-\alpha - 3\beta}(v) = 
h(1,v^{-1}) x_{\alpha+ 3\beta}(-v) \varphi_{\alpha+3\beta}
\begin{pmatrix}
     0 & 1 \\ -1 &v^{-1}
\end{pmatrix}.
$$
Now, $\begin{pmatrix}
     0 & 1 \\ -1 &v^{-1}
\end{pmatrix}\in G_2(\Z_p)$ and 
$h(1,v ) n^-(-x,y,z,u,0) h(1,v^{-1})
= n^-(-vx,y, v^{-1}z, v^{-1}u, 0),$
so 

$$\begin{aligned}
    I_-&=\int\limits_{\Q_p-\Z_p}\int\limits_{\Q_p}\int\limits_{\Q_p}\int\limits_{\Q_p}\int\limits_{\Q_p}f^\circ_{s,p}
(h(1,v^{-1})n^-(-vx,y,v^{-1}z,v^{-1}u,0) x_{\alpha+3\beta}(-v)) \psi ((0,u,y,x) \cdot \sigma) \, dx \,dy\, dz\, du,\\
&=\int\limits_{\Q_p-\Z_p}\int\limits_{\Q_p}\int\limits_{\Q_p}\int\limits_{\Q_p}\int\limits_{\Q_p}
|v|^{-3s+1}
f_{s,p}^\circ
(n^-(-x,y,z,u,0) x_{\alpha+3\beta}(-v)) \psi ((0,uv,y,x/v) \cdot \sigma) \, dx \,dy\, dz\, du,\end{aligned}$$
For the second identity, we used \eqref{eq:ind rep eqvar prop}
and made changes of variable in $x,u,$ and $z.$

Moreover, 
$$
n(-x,y,z,u,0) x_{\alpha+3\beta }(-v) 
= x_{\alpha+3\beta}(-v)x_\beta(uv)n^-(-x+vz-3uvy+u^3v^2, 
y-u^2v , z-u^3v, u, 0).
$$
Making additional changes of variable, we obtain 
$$I_-=
\int\limits_{\Q_p-\Z_p}\int\limits_{\Q_p}\int\limits_{\Q_p}\int\limits_{\Q_p}\int\limits_{\Q_p}
|v|^{-3s+1}
f_{s,p}^\circ
(n^-(-x,y,z,u,0) ) \psi ((0,uv,y+u^2v,x/v+z-3uy-3u^3v) \cdot \sigma) \, dx \,dy\, dz\, du.
$$
The next step is to split each of the integrals 
$I_+$ and $I_-$ into two pieces based on whether $u \in \Z_p$
or $u \in \Q_p-\Z_p.$ After that, we can split 
on $z$ and $y,$ eventually obtaining an expression of the 
original $I^\sigma_p(s, f_{s,p})$ as a sum of sixteen integrals, 
$I_{++++}$ to $I_{----}$
each of which involves only 
$f_{s,p}^\circ(n^-(-x,0,0,0,0)),$ at the expense of having a
more complicated expression inside of $\psi.$
For example $I_{++++}$ corresponds to taking all variables 
in $\Z_p,$ $I_{+-++}$ corresponds to all in $\Z_p$ except $u,$
which is in $\Q_p-\Z_p,$ and so on. 

\subsection{The results of Pleso} 
By restricting their attention to cubic extensions which were
obtained by adjoining a cube root, Jiang and Rallis could 
benefit in two ways. 
First, they could assume that, in the case when 
$\sigma$ was attached to a cubic extension, it was of the form 
$(1,0,0,a).$ This simplified the argument of $\psi$ because 
some of the complicated expressions that would otherwise 
have entered were multiplied by $0.$ So, each 
of their sixteen integrals was simpler than it would have been 
if they'd considered the more general form $(1,0,b,a).$
Second, the assumption assures that certain finite fields which arise in their 
calculations contains three cube roots of $1,$ and they were able 
to use this to their advantage in computing some of the 
sixteen sub-integrals. 

In order to extend the results of Jiang and Rallis to the 
case when the roots of the  polynomial $g_{\sigma}$ attached 
to $\sigma$ generate a cubic extension which can't be 
generated by adjoining a cube root, one has to do two things. 
First, one must re-do the bifurcation process, keeping track 
of the additional complexity that appears in the argument 
of $\psi$ when $\sigma = (1,0,b,a)$ with $b \ne 0.$
Second, one must compute the sixteen generalized sub-integrals 
thus obtained. In some cases, this is a straightforward matter
of checking that the arguments given by Jiang and Rallis in the 
$b=0$ case still work. But in other cases, new ideas are 
required. 

\section{Some technical lemmas} 

Throughout our computations, we assume that $\psi$ is a continuous 
homomorphism $\Q_p \to \C^\times$ which is trivial 
on $\Z_p$ but not on $p^{-1} \Z_p.$ For simplicity, the reader may assume that it is 
$e_p.$ 

The following are some technical lemmas that are applicable to multiple cases in the computations. 

\maketitle
\newcommand{\Int}{\int\limits}

\begin{lem}\label{lem4}
If $c \in \mathbf{C}$, $a \in \mathbf{Q}_p$, $k \in \mathbf{Z}$ are constants. Then

\begin{equation*}
    \Int_{\{x: |x-a|_p \leq p^{-k}\}} c \hspace{0.5cm} \,dx  = cp^{-k}.
\end{equation*}
\end{lem}
\begin{proof}
Indeed, $\{x: |x-a|_p \leq p^{-k}\}$ is precisely the open ball 
$a + p^k \Z_p.$ As discussed in the Haar measure section, this ball 
has measure $p^{-k}.$ And, as with Lebesgue measure the integral of 
a constant function over a measurable set is simply the constant times the 
measure of the set. 
\end{proof}

\begin{lem}[Change of Variables]\label{lem2}
Take $v \in \Q_p^\times,$ and $c \in \Q_p,$ and let $f: \Q_p \to \C$ be an 
integrable function. Then the functions $g,h$ defined by 
$g(x) = f(vx)$ and $h(x) = f(x+c)$ are also integrable, and their integrals
are given by 
$$
\int_{\Q_p} f(vx) \, dx = |v|_p^{-1} \int_{\Q_p} f(z) \, dz, 
$$
and 
$$
\int_{\Q_p} f(x+c) \, dx = \int_{\Q_p}f(y) \, dy.
$$
This is expressed by saying that if $z=vx$ then $dz = |v|_p \, dx$ and 
$dy = dx.$ 
\end{lem}
\begin{proof}
    The integral of an arbitrary measurable function is defined by approximating it with step functions as in \eqref{eq:stepFunction}, so it suffices to treat the 
    case when $f$ is a step function. Since both sides of each formula are linear in 
    $f$ it suffices to treat the case when $f$ is the characteristic function 
    $\mathbbm{1}_{a+p^k \Z_p}$ for some $a,k.$ Clearly, $x+c \in a+p^k \Z_p$ if and only if 
    $x \in a-c+p^k \Z_p.$ Then 
    $$
\int_{\Q_p} f(x+c) \, dx =
\int_{\Q_p} \mathbbm{1}_{a-c+p^k} = \on{Vol}(a-c+p^k)=p^{-k},$$
and 
$$
\int_{\Q_p}f(y) \, dy = \int_{\Q_p} \mathbbm{1}_{a+p^k} = \on{Vol}( a+p^k) = p^{-k}.
$$
Notice that this is just the invariance property of the Haar measure. 
Similarly, $vx \in a+p^k\Z_p$ if and only if $x \in v^{-1} a + p^{k-\ord(v)} \Z_p.$ 
Thus 
$$\int_{\Q_p} f(vx) \, dx= \on{Vol}(v^{-1} a + p^{k-\ord(v)} \Z_p)
= p^{\ord(v) -k} = |v|_p^{-1} \int_{\Q_p}f(z) \, dz.
$$
\end{proof}

\begin{lem}\label{lem:char(a)}
    $$
    \Int_{\Z_p} \psi( ax) \, dx = \mathbbm{1}_{\Z_p}(a).$$
\end{lem}
    This lemma is quite well-known, but the proof is short and 
    nice, so we include it. 
\begin{proof}
If $a \in \Z_p$ then 
$ax \in \Z_p$ for all $x.$ But then $\psi(ax) = 1$ for all $x,$
and we just get the measure of $\Z_p,$ which is $1.$ 
If $a \notin \Z_p$ then $\psi$ is not trivial on 
$a\Z_p.$ Choose $x_0$ such that $\psi(ax_0) \ne 1.$ Then 
$$
\Int_{\Z_p} \psi( ax) \, dx
= \Int_{\Z_p} \psi( a (x+x_0)) \, dx 
= \psi(ax_0) \Int_{\Z_p} \psi( ax) \, dx,
$$
which forces the integral to be $0.$
\end{proof}

\begin{lem}\label{lem3}
Let $\mathbbm{1}_{\mathbf{Z}_p}(t)$ be a characteristic function that is 1, if $t \in \mathbf{Z}_p$ or 0 if $t \notin \mathbf{Z}_p$. Then

\begin{equation*}
    \Int_{\mathbf{Z}_p^*} \psi(-xt) \,dx = \mathbbm{1}_{\mathbf{Z}_p}(t) - p^{-1}\mathbbm{1}_{\mathbf{Z}_p}(pt).
\end{equation*}
\end{lem}
\begin{proof}
    Express the left hand side as the 
    integral over $\Z_p$ minus the integral over $p\Z_p.$ Apply lemma \ref{lem:char(a)} to 
    the first integral. In the second integral, use lemma \ref{lem2} to substitute $y=px.$ As $x$ ranges over 
    $p\Z_p,$ $y$ ranges over $\Z_p,$ and so lemma \ref{lem:char(a)} can be applied again.
\end{proof}
\begin{rmk}
    We have included the minus sign in lemma \ref{lem3} because a minus sign is present 
    in some of the applications of this lemma which appear below. But the 
    same formula holds true if the minus sign is omitted, since $\mathbbm{1}_{\mathbf{Z}_p}$ is an even function.
\end{rmk}

\begin{cor} \label{cor1}
Let $a,b \in \mathbf{Z}$, $c \in \Q_p$, $u \in \mathbf{Q}_p - \mathbf{Z}_p$ and $u = \mu p^U$. Then

\begin{equation*}
\begin{aligned}
    &\Int_{\mathbf{Q}_p-\mathbf{Z}_p} |u|_p^{as + b} \psi(cu) \,du =  \qquad \sum_{U=-\infty}^{-1} p^{-U(as+b+1)} \ \Int_{\mathbf{Z}_p^*} \psi(c\mu p^U) \,d\mu\\
    & \begin{cases}
        \frac{p^{as+b}}{1-p^{as+b}}\left[
        (1-|c|_p^{-as-b-1})-p^{-1}(1-|c|_p^{-as-b-1}p^{as+b+1})
        \right], & c \in \Z_p, \\ 0, & c \notin \Z_p.
    \end{cases}
 \end{aligned}   
\end{equation*}
\end{cor}
\begin{proof}
    The first equation holds because $\Q_p - \Z_p$ is the countable disjoint union of the 
    sets $p^U\Z_p^*,$ for $U$ from $-1$ to $-\infty.$ We plug in $u= p^U \mu$ and use 
    lemma \ref{lem2}. Each of the $\mu$ integrals can then be computed using lemma 
    \ref{lem3}, yielding 
    $$
    \sum_{U=-\infty}^{-1} p^{-U(as+b+1)} 
    (\mathbbm{1}_{\Z_p} (cp^{U}) -p^{-1} \mathbbm{1}_{\Z_p} (cp^{U+1})).
    $$
    Since $cp^U$ (respectively, $cp^{U+1}$) is in $\Z_p$ if and only 
    if $U \ge -\ord(c)$ (respectively $-\ord(c) - 1$), we obtain 
    $$
    \sum_{U = -\ord(c)}^{-1} p^{-U(as+b+1)}  - p^{-1} \sum_{U = -\ord-1}^{-1} p^{-U(as+b+1)}.
    $$
    Summing these finite geometric series yields
    $$
    \frac{p^{(as+b+1)}- p^{(as+b+1)(\ord(c)+1)}}{1-p^{as+b+1}}
    - p^{-1}  \frac{p^{(as+b+1)}- p^{(as+b+1)(\ord(c)+2)}}{1-p^{as+b+1}}.
    $$
    Plugging in $|c|_p = p^{-\ord(c)}$ and simplifying gives the second identity.
\end{proof}

\begin{lem}\label{lem1}
For $s \in \mathbf{C}$, if $a \in \mathbf{Z}_p$, then

\begin{equation*}
    \Int_{\mathbf{Q}_p}f_s^\circ(n^-(-x))\psi(ax) \,dx = \frac{1-p^{-3s}}{1-p^{1-3s}}(1-|a|_p^{3s-1}p^{1-3s}). 
\end{equation*}

Otherwise, the integral vanishes.
\end{lem}
\begin{proof}
We split the domain of integration into two pieces:
$\Z_p$ and $\Q_p-\Z_p.$ If $x \in \Z_p$
then $n^-(-x) \in G_2(\Z_p)$ and 
hence $f_s^\circ(n^-(-x))=1.$ By Lemma \ref{lem:char(a)}, 
we get 
$$\Int_{\mathbf{Z}_p}f_s^\circ(n^-(-x))\psi(ax) \,dx=
\mathbbm{1}_{\Z_p})(a) \, dx.$$
If 
$|x|_p > 1$ then we use the identity \eqref{SL2Iwasawa}
and an embedding $SL_2 \to G_2$ to check that 
 $f_s^\circ(n^-(-x)) = |x|_p^{-3s}.$ We can then 
 apply corollary \ref{cor1} to compute this part. 
 
Combining the two parts gives
\begin{equation*}
    \Int_{\mathbf{Q}_p}f_s(n^-(-x))\psi(ax) \,dx = \frac{p^{-(3s-1)}-p^{-(3s-1)}|a|^{3s-1}}{1-p^{-(3s-1)}} - p^{-3s}|a|^{3s-1} + 1. 
\end{equation*}

Now we find a common denominator, expand and simplify the expression to have

\begin{equation*}
    \Int_{\mathbf{Q}_p}f_s(n^-(-x))\psi(ax) \,dx = \frac{1-p^{-(3s-1)}-p^{-3s}|a|^{3s-1} + p^{-(6s-1)}|a|^{3s-1}+ p^{-(3s-1)}-p^{-(3s-1)}|a|^{3s-1}}{1-p^{-(3s-1)}}
\end{equation*}

\begin{equation*}
    = \frac{1-p^{-3s}|a|^{3s-1}+p^{-(6s-1)}|a|^{3s-1} - p^{-(3s-1)}|a|^{3s-1}}{1-p^{-(3s-1)}} = \frac{1-p^{-3s}}{1-p^{1-3s}}(1-|a|_p^{3s-1}p^{1-3s}). 
\end{equation*}
\end{proof}

\begin{lem}\label{lem5}
Let $h\in \Z[r,y,u]$ be a polynomial in three variables
with coefficients in $\Z_p,$ and for any negative integer
$V,$ let 
$S(V) =\{(r,y,u) \in (\frac{\mathbf{Z}}{p^{-V}\mathbf{Z}}) \times (\frac{\mathbf{Z}}{p^{-V}\mathbf{Z}}) \times (\frac{\mathbf{Z}}{p^{-V}\mathbf{Z}}) | h(r,y,u) = 0 \}$. Then, for each $V,$ reduction mod $p^{-V}$ gives a well-defined function $\rho : S(V-1) \rightarrow S(V)$.
\end{lem}
\begin{proof}
    Indeed, if $(r_1,y_1, u_1)$ and $(r_2, y_2, u_2)$ 
    are equivalent modulo $-V+1,$ then they are
    equivalent modulo $-V,$ and if 
    $h(r,y,u)$ is $0$ modulo $-V+1,$ then it is 
    zero modulo $V.$
\end{proof}
\begin{rmk}
    In some of our applications one or more of the 
    variables is restricted to $(\Z/p^{-V}\Z)^*.$
    The lemma remains true in this context with the 
    same proof, since an element of $\Z/p^{-V+1}\Z$ 
    is a unit if and only if its image in $\Z/p^{-V}\Z$
    is.
\end{rmk}

\begin{lem}
\label{lem6}
Let $h, V$ and $S(V)$ be as in lemma \ref{lem5}.
 Let $ \overline{h}$ be the 
 image of $(h_{r}(t),h_{y}(t),h_{u}(t))$
 in $(\Z/p\Z)^3,$ and suppose that it 
 is not zero.
  Here $h_r$ is the partial derivative of $h$ with respect
 to the variable $r$ and so on.
 Then, for all $t \in S(V),$  $$ \# \{a \in S(V-1) | \rho(a) = t \} = p^2.$$
\end{lem}

\begin{proof}
Fix $t=(r_0, y_0,u_0) \in S(V).$
The number of $a \in (\Z/p^{-V+1}\Z)^3$
such that $\rho(a) = t$ is $p^3,$
and we can express each element uniquely as 
 $a = (r_0 + p^{-V} \alpha, y_0 + p^{-V} \beta, u_0 + p^{-V} \gamma),$
 with $\alpha, \beta \gamma \in \Z/p\Z.$
 We can plug in a Taylor expansion in each 
 variable, and higher order terms vanish modulo $p^{-2V}.$
 Hence $$h(a)\equiv h(t) + p^{-V} (\alpha h_{r}(t) + \beta h_{y}(t) + \gamma h_{u}(t))\mod p^{-2V}.$$

 Since $h(r_0, y_0,u_0),$ is divisible by $p^{-V},$ 
 we can write it 
 as $-p^{-V}d,$ for some $d \in \Z_p.$
Then 
$$h(a)  \equiv 0 \mod p^{-V+1} \iff (\alpha h_{r}(t) + \beta h_{y}(t) + \gamma h_{u}(t)) \equiv -d \mod p.$$
This expression can be written as a matrix representation;
\begin{equation*}
    (\overline{h})\begin{pmatrix} \alpha \\ \beta \\ \gamma \end{pmatrix} = d
\end{equation*}
This is a linear equation in three variables
over the field with $p$ elements. If $d=0,$ 
then its solution is a two dimensional 
subspace of $(\Z/p\Z)^3.$ If not, 
it's a coset of that subspace. In either 
case, it has $p^2$ elements. 
 \end{proof}

\begin{cor}\label{cor2}
Let $S(V)$ be defined as in lemma \ref{lem6}, using 
a polynomial $h$ such that $\overline h,$
defined as in lemma \ref{lem6} is nonzero. 
Let $N(V)$ be the number of elements in $S(V).$ Then 
$N(V-1) = p^2N(V)$. 
\end{cor}

\begin{proof}
Notice that by Lemma \ref{lem6}, $\# \{a \in S(V-1) | \rho(a) = t \} = p^2,$ for each element of $S(V).$ So the number of elements in $S(V-1)$ is simply the number of elements in $S(V-1)$ that map to each element in $S(V),$ $p^2$, times the amount of elements in $S(V),$ $N(V).$ Therefore $N(V-1) = p^2N(V)$. 
\end{proof}

\section{Computations}

We compute all 16 pieces of $I^{\sigma}_p(s, f_{s,p}^\circ)$
when $\sigma = (1,0,b,c)$ with $b,c \in \Z_p^\times$ 
such that $g_\sigma(u) = -u^3+bu+c$ is irreducible, and all $p \equiv 5
\mod 6.$ Some of our results do not need all of these conditions 
to be satisfied. Throughout this section, $\sigma$
is fixed, and we denote $g_\sigma$ simply 
as $g.$ 

We remark that the values of $I_{++++}, I_{+++-}, I_{++-+}, I_{++--},
I_{+-+-},I_{+--+},I_{+---},I_{-++-},$ and 
$I_{-+--}$ were already computed by 
Pleso in \cite{pleso2009integrals}. We include them here for the sake
of completeness. The first four may reasonably be attributed to 
Jiang and Rallis in \cite{Jiang-Rallis}, since, as Pleso points out, 
the coefficient $b$, which is zero in Jiang-Rallis and nonzero here, 
does not appear at all in them.

Case 1 ($I_{++++}$)

\begin{equation*}
    \mathbf{I}_{++++} = 
    \Int_{\mathbf{Q}_p}   f_s(n^-(-x)) \psi(x)\,dx.  
\end{equation*}  

Applying Lemma \ref{lem1}, notice that since $|1|_p = 1, 1 \in \mathbf{Z}_p$, so the solution to the integral is as follows:
\begin{equation*}
    \mathbf{I}_{++++} = 
    \Int_{\mathbf{Q}_p}   f_s(n^-(-x)) \psi(x)\,dn  = \frac{1-p^{-3s}}{1-p^{-3s+1}}(1-|1|^{3s-1}p^{-3s+1}) = 1-p^{-3s}.
\end{equation*}

Case 2 ($I_{+++-}$)

\begin{equation*}
    \mathbf{I}_{+++-}  
    = \Int_{\mathbf{Q}_p-\mathbf{Z}_p} |y|_p^{-9s+3} \Int_{\mathbf{Q}_p} f_s(n^-(-x)) \psi(xy^3)\,dx \,dy.
\end{equation*}  

Notice that since $|y|_p > 1$, then $|y^3|_p > 1,$ which implies that $y^3 \notin \mathbf{Z}_p$. By Lemma \ref{lem1}, the integral vanishes.

Case 3 ($I_{++-+}$)

\begin{equation*}
\mathbf{I}_{++-+} = 
    \Int_{\mathbf{Q}_p-\mathbf{Z}_p} \Int_{\mathbf{Z}_p}|z|_p^{-6s+2} \psi(-z^2y^3)\Int_{\mathbf{Q}_p} f_s(n^-(-x)) \psi(xz)\,dx \,dy \,dz.
\end{equation*}

Notice that since $|z|_p > 1$ then $z \notin \mathbf{Z}_p$. By Lemma \ref{lem1}, the integral vanishes.

Case 4 ($I_{++--}$)

\begin{equation*}
\mathbf{I}_{++--} = 
   \Int_{\mathbf{Q}_p-\mathbf{Z}_p} \Int_{\mathbf{Q}_p-\mathbf{Z}_p} f_s(n^-(-x)) |y|_p^{-9s+3} |z|_p^{-6s+2} \psi(-z^2y^3) \Int_{\mathbf{Q}_p} \psi(xy^3z)\,du \,dy \,dz.
\end{equation*}

Notice that since $|y|_p > 1$ and $|z|_p > 1$, then $|y^3z|_p > 1$ which implies that $y^3z \notin \mathbf{Z}_p$. By Lemma \ref{lem1}, the integral vanishes. 

Case 5 ($I_{+-++}$)

\begin{equation*}
\mathbf{I}_{+-++}
    =\Int_{\mathbf{Q}_p-\mathbf{Z}_p} \Int_{\mathbf{Z}_p} \Int_{\mathbf{Z}_p}  |u|_p^{-9s+4} \psi(bu-u^3z^2 -3uy^2-3u^2yz) \Int_{\mathbf{Q}_p} f_s(n^-(-x)) \psi(x)\,dx \,dz \,dy \,du.
\end{equation*}

Notice that the integral with respect to $x$ is identical to the integral with respect to $x$ in Case 1 ($I_{++++}$). Let $\psi_1(a,d) = a^2 +3ad + 3d^2$. Then,

\begin{equation*}
\mathbf{I}_{+-++} = 
    (1-p^{-3s}) \Int_{\mathbf{Q}_p-\mathbf{Z}_p} \Int_{\mathbf{Z}_p} \Int_{\mathbf{Z}_p}  |u|_p^{-9s+4} \psi(bu-u\psi_1(uz,y))\,dz \,dy \,du.
\end{equation*}

Applying Corollary \ref{cor1} we get,

\begin{equation*}
    \mathbf{I}_{+-++} =  (1-p^{-3s}) \qquad \sum_{U=-\infty}^{-1} p^{-U(-9s+5)}\  \Int_{\mathbf{Z}_p^*} \Int_{\mathbf{Z}_p} \Int_{\mathbf{Z}_p} \psi(b\mu p^{U}- \mu p^U\psi_1(\mu p^Uz,y)) \,dz \,dy \,d\mu.
\end{equation*}

Let $w = \mu p^Uz \rightarrow z = \mu^{-1}p^{-U}w$. Then, $z \in \mathbf{Z}_p \iff w \in p^U\mathbf{Z}_p$. Also, by Lemma \ref{lem2} $\,dz = p^U \,dw$. So, 

\begin{equation*}
    \mathbf{I}_{+-++} =  (1-p^{-3s}) \qquad \sum_{U=-\infty}^{-1} p^{-U(-9s+4)}\  \Int_{\mathbf{Z}_p^*} \Int_{\mathbf{Z}_p} \Int_{p^U\mathbf{Z}_p} \psi(\mu p^U (b - \psi_1(w,y))) \,dw \,dy \,d\mu.
\end{equation*}

Applying Lemma \ref{lem3},

\begin{equation*}
    \mathbf{I}_{+-++} = (1-p^{-3s}) \qquad \sum_{U=-\infty}^{-1} p^{-U(-9s+4)}\  \Int_{\mathbf{Z}_p} \Int_{p^U\mathbf{Z}_p} \mathbbm{1}_{\mathbf{Z}_p}(p^U(b - \psi_1(w,y)))
\end{equation*}

\begin{equation*}
    - p^{-1} \mathbbm{1}_{\mathbf{Z}_p}(p^{U+1}(b - \psi_1(w,y))) \,dw \,dy.
\end{equation*}
For $k \in \mathbf{Z}, k \geq 0,$ let
\begin{equation*}
    S_{k,b} = \{(w,y) \in \mathbf{Z}_p^2: b - \psi_1(w,y)  \in p^k\Z_p\}.
\end{equation*}
Notice that, for any nonpositive integer $U,$
$$S_{k,b} = \{(w,y) \in p^U \Z_p \times \mathbf{Z}_p: b - \psi_1(w,y) \in p^k \Z_p\}.$$
That is, allowing $w$ to range over the larger 
set $p^U\Z_p$ does not add any elements
for if $\ord(w) = W < 0,$ then 
$\psi_1(w,y),$ and hence also $b-\psi_1(w,y),$ will have order $2W<0
\le k.$
Now, the set $S_{k,b}$ is a union of $p^k \Z_p\times p^k \Z_p$ cosets, since 
the truth or falsity of the statement ``$b - \psi_1(w,y) \in p^k \Z_p$''
only depends on the image of $w,$ and $y$ in the quotient ring
$\Z_p /p^k \Z_p =\Z/p^k \Z.$
Thus if $$V(k,b) := Vol(S_{k,b}) = \Int_{\mathbf{Z}_p} \Int_{\mathbf{Z}_p} \mathbbm{1}_{\mathbf{Z}_p}(p^{-k}(b - \psi_1(w,y)) \,dw \,dy ,$$
then for $k\ge 1,$
\begin{equation}\label{eq:V(k,b)}
\begin{aligned}
    V(k,b) &= p^{-2k} \# \{(w_1,y_1) \in (\frac{\mathbf{Z}}{p^k\mathbf{Z}})^2 : \psi_1(w,y) = b \mod p^k\},
    \end{aligned}
\end{equation}
since each coset of $p^k \Z_p\times p^k \Z_p$
has volume $p^{-2k}.$
Further, 
\begin{equation*}
     \mathbf{I}_{+-++}=(1-p^{-3s}) \qquad \sum_{U=-\infty}^{-1} p^{-U(-9s+4)}\ \ (V(-U,b) - p^{-1}V(-U-1,b)).
\end{equation*}

\begin{prop}\label{prop0}
Let $p$ be a prime such that $p \equiv 5 \pmod{6}.$ Then 
for any positive integer $k$ 
$\# \{(w_1,y_1) \in (\frac{\mathbf{Z}}{p^k\mathbf{Z}})^2 : \psi_1(w,y) = b\} = (p+1)p^{k-1}.$
\end{prop}
\begin{proof}
    Since $p \equiv 5 \pmod{6},$ the finite field $\Z/p\Z$
    contains no nontrivial cube roots of $1.$ so the polynomial 
    $x^2+x+1$ is irreducible over $\Z/p\Z$ and a zero $\alpha$
    of it generates the unique quadratic extension $\F_{p^2}.$ 
    Note that if $\alpha$ is one of the zeros of 
    $x^2+x+1,$ then the other one is $\alpha^p,$ which is also 
    $\alpha^2,$ because $\alpha^3=1$ and $p \equiv 2 \mod 3.$
    Each element of $\F_{p^2}$ has the form $c+d\alpha$ for 
    $c,d \in \Z/p\Z,$ and the {\bf norm} 
    $$N(c+d\alpha ) = (c+d\alpha)(c+d\alpha^p)=(c+d\alpha)^{p+1}=c^2-cd+d^2$$
    maps $\F_{p^2}$ to $(\Z/p\Z)$  For $y,w \in \Z/p\Z,$ we can 
    now recognize $\psi_1(y,w)$ as the norm of $y+w-w\alpha.$ 
    So, when $k=1$ we are counting the number of elements of
    $\F_{p^2}$ of norm $b.$
    Since $\F_{p^2}^*$ is a cyclic group of order $p^2-1,$
    whose unique subgroup of order $p-1$ is $(\Z/p\Z)^*,$ and since $N(\xi)=\xi^{p+1},$ it follows that each element of $(\Z/p\Z)^*,$ 
    is the norm of precisely $p+1$ elements of $\F_{p^2}^*.$ 
    This completes the case $k=1.$ 

    We now proceed by induction. Note that 
    $$\psi_1(w_1+w_2p^k, y_1+y_2p^k)
    \equiv \psi_1(w_1,w_2)+p^k((2w_1+3y_1)w_2+(3w_1+6y_1)y_2) \mod p^{k+1}.$$
    Also, if $\psi_1(w_1, y_1) \equiv b \ne 0 \mod p^k $
    with $k \ge 1,$ then $(w_1, y_1) \neq 0 \mod p.$ But then 
    $(2w_1+3y_1, 3w_1+6y_1) \neq 0 \mod p,$ since $\bsbm 2&3\\ 3&6 \esbm$ is nonsingular $\mod p.$ But then 
    $$(w_2, y_2) \mapsto (2w_1+3y_1)w_2+(3w_1+6y_1)y_2$$
    is a nontrivial linear map $\Z/p\Z^2 \to \Z/p\Z,$ so it is 
    surjective and each of its fibers has $p$ elements. 
    It follows that the number of
    solutions to $\psi_1(w,y) = b \mod p^{k+1}$ is precisely
    $p$ times the number of solutions mod $k.$ This completes the proof.
    \end{proof}

Combining proposition \ref{prop0} with equation 
\eqref{eq:V(k,b)}, and noting that when $k=0,$ the set
$S_{k,b}$ is clearly all of $\Z_p^2,$ which has volume one, 
we obtain.
\begin{equation*}
V(k,b) =
\left\{
   \begin{array}{cc}
      1, & k = 0, \\
      (1+p^{-1})p^{-k}, & k \geq 1. \\
   \end{array}\right.
\end{equation*}

Hence, $V(k,b) - p^{-1}V(k-1,b) $ is equal to 

\begin{equation*}
\left\{
   \begin{array}{cc}
      (1+p^{-1})p^{-1} - p^{-1}(1) = p^{-1}+p^{-2}-p^{-1} = p^{-2},& k = 1, \\
      (1+p^{-1})p^{-k} - p^{-1}(1+p^{-1})p^{-k+1} = (1+p^{-1})(p^{-k} - p^{-k+1-1}) = 0, & k \geq 2. \\
   \end{array}\right.
\end{equation*}

Going back to the integral, we now know $V(-U,b) - p^{-1}V(-U-1,b) = p^{-2}$ if $U = -1,$ and zero otherwise.

Therefore,
\begin{equation*}
    \mathbf{I}_{+-++} = (1-p^{-3s})p^{-9s+4}p^{-2} =  (1-p^{-3s})p^{-9s+2}.
\end{equation*}

Case 6 ($I_{+-+-}$)

\begin{equation*}
\mathbf{I}_{+-+-} 
    =\Int_{\mathbf{Z}_p} \Int_{\mathbf{Z}_p} \Int_{\mathbf{Q}_p-\mathbf{Z}_p}|y|_p^{-9s+3} |u|_p^{-9s+4} \psi(-3uy(y+uz)-u^3z^2) \Int_{\mathbf{Q}_p} f_s(n^-(-x)) \psi(xy^3)\,dx \,du \,dy \,dz.
\end{equation*}

Notice that the integral with respect to $x$ is identical to the integral with respect to $x$ in Case 2 ($I_{+++-}$), and we've proven it vanishes. Therefore the integral vanishes.

Case 7 ($I_{+--+}$)

\begin{equation*}
\mathbf{I}_{+--+}
    =\Int_{\mathbf{Z}_p} \Int_{\mathbf{Q}_p-\mathbf{Z}_p} \Int_{\mathbf{Q}_p-\mathbf{Z}_p}|u|_p^{-9s+4} |z|_p^{-6s+2} \psi(-u^3z^2 - y^3z^2 - 3u^2yz^2- 3uy^2z^2) \Int_{\mathbf{Q}_p} f_s(n^-(-x)) \psi(xz)\,dx \,du \,dy \,dz.
\end{equation*}

Notice that the integral with respect to $x$ is identical to the integral with respect to $x$ in Case 3 ($I_{++-+}$), and we've proven it vanishes. Therefore the integral vanishes.

Case 8 ($I_{+---}$)

\begin{equation*}
\mathbf{I}_{+---}
    =\Int_{\mathbf{Q}_p-\mathbf{Z}_p} \Int_{\mathbf{Q}_p-\mathbf{Z}_p}\Int_{\mathbf{Q}_p-\mathbf{Z}_p} |u|_p^{-9s+4} |z|_p^{-6s+2} |y|_p^{-3s+1} \psi(-u^3z^2- y^3z^2- 3u^2yz^2- 3uy^2z^2) 
\end{equation*}

\begin{equation*}
    \Int_{\mathbf{Q}_p} f_s(n^-(-x)) \psi(xy^3z)\,dx \,du \,dy \,dz.
\end{equation*}

Notice that the integral with respect to $x$ is identical to the integral with respect to $x$ in Case 3 ($I_{++--}$), and  we've proven it vanishes. Therefore the integral vanishes.

Case 9 ($I_{-+++}$)

\begin{equation*}
\mathbf{I}_{-+++} 
    = \Int_{\mathbf{Q}_p-\mathbf{Z}_p} \Int_{\mathbf{Z}_p} |v|_p^{-3s+1} \psi(cv- u^3v +buv) \Int_{\mathbf{Q}_p} f_s(n^-(-x)) \psi(xv^{-1})\,dx \,du \,dv.
\end{equation*}

Notice that since $|v|_p > 1$, then $\ord_pv < 0$. Considering $v^{-1}$, notice $|v^{-1}| = p^{\ord_pv} < 1$. So this implies that $v^{-1} \in \mathbf{Z}_p$. By Lemma \ref{lem1}, 

\begin{equation*}\begin{aligned}
  \mathbf{I}_{-+++}   &=(\frac{1-p^{-3s}}{1-p^{-3s+1}}) \Int_{\mathbf{Q}_p-\mathbf{Z}_p} \Int_{\mathbf{Z}_p} |v|_p^{-3s+1} (1-|v^{-1}|^{3s-1}p^{-3s+1}) \psi(cv -u^3v+buv)  \,du \,dv
   \\& =(\frac{1-p^{-3s}}{1-p^{-3s+1}}) \Int_{\mathbf{Q}_p-\mathbf{Z}_p}  \Int_{\mathbf{Z}_p} (|v|^{1-3s}-|v|^{2-6s}p^{-3s+1}) \psi(cv -u^3v+buv) \,du \,dv.
\\&=\frac{1-p^{-3s}}{1-p^{-3s+1}}[H(1-3s) - p^{1-3s}H(2-6s)], \end{aligned}
\end{equation*}
where
\begin{equation*}
    H(\lambda) = \Int_{\mathbf{Q}_p-\mathbf{Z}_p}  \Int_{\mathbf{Z}_p} |v|^{\lambda}\psi(cv -u^3v+buv) \,du \,dv.
\end{equation*}

Applying Corollary \ref{cor1},
\begin{equation*}
     H(\lambda) = \Int_{\mathbf{Z}_p} \qquad \sum_{V=-\infty}^{-1} p^{-V(\lambda+1)}\  \Int_{\mathbf{Z}_p^*} \psi(p^V\nu(c-u^3+bu)) \,d\nu \,du.
\end{equation*}

Applying Lemma \ref{lem3},
\begin{equation*}
    H(\lambda) = \qquad \sum_{V=-\infty}^{-1} p^{-V(\lambda+1)}\ \Int_{\mathbf{Z}_p} \mathbbm{1}_{\mathbf{Z}_p}(p^V(u^3-bu-c))-p^{-1}\mathbbm{1}_{\mathbf{Z}_p}(p^{V+1}(u^3-bu-c)) \,du.
\end{equation*}

For a nonpositive integer $V,$ let 

\begin{equation*}
    h(V) = \Int_{\mathbf{Z}_p} \mathbbm{1}_{\mathbf{Z}_p}(p^V(u^3-bu-c)) \,du  
    =Vol(\{u \in \mathbf{Z}_p: |u^3-bu-c|_p \leq p^{V}\}).
\end{equation*}

Then,

\begin{equation*}
    H(\lambda) = \qquad \sum_{V=-\infty}^{-1} p^{-V(\lambda+1)}\ (h(V)-p^{-1}h(V+1)).
\end{equation*}

By Lemma \ref{lem4},
\begin{equation*}
h(V) =
\left\{
   \begin{array}{cc}
      Vol(\mathbf{Z}_p) = 1, & V = 0, \\
      N(b,c)p^{V}, & V< 0, \\
   \end{array}\right.
\end{equation*}
where $N(b,c) \in \{0,1,3\}$ is the number of zeros of $g(u) = -u^3 +bu +c$ mod $p.$ 
(And, in particular, $N(b,c) = 0$ when $g$ is irreducible.
Thus, 
$$
h(V)-p^{-1}h(V+1)=\begin{cases}
    N(b,c)p^{-1}-p^{-1}, & V=-1, \\
    0, & V< -1, 
\end{cases}
$$
and hence 

\begin{equation*}
\begin{aligned}
    H(\lambda) &= p^{\lambda +1}(N(b,c)-1)p^{-1} 
\\
    &=p^{\lambda}(N(b,c)-1).
\end{aligned}
\end{equation*}

Therefore,

\begin{equation*}
    \mathbf{I}_{-+++} = \frac{1-p^{-3s}}{1-p^{-3s+1}}[p^{-3s+1}(N(b,c)-1) - p^{-3s+1}p^{-6s+2}(N(b,c)-1)]
\end{equation*}

\begin{equation*}
   = (1-p^{-3s})(N(b,c) - 1)(\frac{1 - p^{-6s+2}}{1-p^{-3s+1}})p^{-3s+1} = (1-p^{-3s})(N(b,c)-1)(1+p^{-3s+1})p^{-3s+1}     
\end{equation*}
\begin{equation*}
    = (1-p^{-3s})(N(b,c)-1)(p^{-3s+1}+p^{-6s+2}).
\end{equation*}
In particular, when $g$ is irreducible, 
$$I_{-+++} = -p^{-3s+1}(1-p^{-3s})(1+p^{-3s+1}).$$

Case 10 ($I_{-++-}$)

\begin{equation*}
\mathbf{I}_{-++-} =
    \Int_{\mathbf{Q}_p - \mathbf{Z}_p} \Int_{\mathbf{Q}_p - \mathbf{Z}_p} \Int_{\mathbf{Z}_p} |v|_p^{-3s+1} |y|_p^{-9s+3}\psi(cv- u^3v-3uy+ buv) \Int_{\mathbf{Q}_p} f_s(n^-(-x))\psi(xv^{-1}y^3) \,dx \,du \,dv \,dy.
\end{equation*}

If $v^{-1}y^3 \notin \mathbf{Z}_p$ then the $x$ integral vanishes. 
So Lemma \ref{lem1}, yields a new integral where the 
condition $|y^3|_p \le |v|_p$ must be incorporated into the 
domain of integration: 
\begin{equation*}
\begin{aligned}
    \mathbf{I}_{-++-} &= 
    \frac{1-p^{-3s}}{1-p^{-3s+1}}\Int_{\mathbf{Z}_p}\iint\limits_{|v|_p \geq |y^3|_p>1}  |v|_p^{-3s+1} |y|_p^{-9s+3}\psi(cv+buv- u^3v-3uy)(1-|v^{-1}y^3|_p^{3s-1}p^{-3s+1})  \,dv \,dy
\, du\\ &
    =\frac{1-p^{-3s}}{1-p^{-3s+1}}(H(1-3s,3-9s)-p^{1-3s}H(2-6s,0)),\end{aligned}
\end{equation*}
where
\begin{equation*}
    H(\lambda,\mu) =  \Int_{\mathbf{Z}_p} \iint\limits_{|v|_p \geq |y^3|_p>1} |y|_p^\mu|v|_p^\lambda \psi(-3uy+(c+ bu- u^3)v) \,dv \,dy \,du.
\end{equation*}

Applying Corollary \ref{cor1} in both $y$ and $v$,
\begin{equation*}
    H(\lambda,\mu) = \Int_{\mathbf{Z}_p} \mathop{\sum_{Y=-\infty}^{-1}\sum_{V=-\infty}^{3Y}}  p^{Y(\mu+1)+V(\lambda+1)} \Int_{\mathbf{Z}_p^*} \psi(3up^Y\gamma) \,d\gamma \Int_{\mathbf{Z}_p^*} \psi((c+ bu- u^3)p^V\nu) \,d\nu \,du.
\end{equation*}

Applying Lemma \ref{lem3}
\begin{equation*}
   H(\lambda,\mu)  = \mathop{\sum_{Y=-\infty}^{-1}\sum_{V=-\infty}^{3Y}}  p^{Y(\mu+1)+V(\lambda+1)}  \Int_{\mathbf{Z}_p} (\mathbbm{1}_{\mathbf{Z}_p}(3up^Y) - p^{-1}\mathbbm{1}_{\mathbf{Z}_p}(3up^{Y+1}) 
\end{equation*}

\begin{equation*}
    (\mathbbm{1}_{\mathbf{Z}_p}((c+ bu- u^3)p^V) - p^{-1}\mathbbm{1}_{\mathbf{Z}_p}((c+ bu- u^3)p^{V+1})\,du.
\end{equation*}








We now make use of the assumption that 
$g(u)$ is irreducible$\mod p$. 
It follows from this assumption that $g(u) \in \Z_p^\times$ for
all $u \in \Z_p.$
Also, it follows that 
$1_{\Z_p}(g(u) p^V) = 1_{\Z_p}(g(u)p^{v+1})=0$ for all $V< 3Y \le -3.$

Therefore $ H(\lambda,\mu) = 0$, so the integral vanishes.

Case 11 ($I_{-+-+}$) 

\begin{equation*}
    \mathbf{I}_{-+-+} =  \Int_{\mathbf{Q}_p - \mathbf{Z}_p} \Int_{\mathbf{Z}_p}  \Int_{\mathbf{Q}_p - \mathbf{Z}_p} \Int_{\mathbf{Z}_p} |v|_p^{-3s+1} |z|_p^{-6s+2}\psi(z+cv-u^3v+buv-3uyz-v^{-1}y^3z^2)  
\end{equation*}

\begin{equation*}
    \Int_{\mathbf{Q}_p} f_s(n^-(-x))\psi(v^{-1}xz) \,dx \,du \,dv \,dy \,dz.
\end{equation*}

 By Lemma \ref{lem1},
\begin{equation*}\begin{aligned}
    \mathbf{I}_{-+-+} =\frac{1-p^{-3s}}{1-p^{-3s+1}} \iint\limits_{|v|_p\geq|z|_p>1} \Int_{\mathbf{Z}_p} \Int_{\mathbf{Z}_p} & |v|_p^{-3s+1} |z|_p^{-6s+2}
    (1-|v^{-1}z|_p^{3s-1}p^{-3s+1}) 
   \\ &\psi(z+cv+buv- u^3v-3uyz -v^{-1}y^3z^2)\,du \,dy \,dv \,dz .\end{aligned}
\end{equation*}

Let $z = rv.$ Then by Lemma \ref{lem3}, $dz = |v|_p dr$. So

\begin{equation*}
\begin{aligned}    
   \mathbf{I}_{-+-+} =\frac{1-p^{-3s}}{1-p^{-3s+1}} \iint\limits_{|v^{-1}|_p<|r|_p\leq 1} \Int_{\mathbf{Z}_p} \Int_{\mathbf{Z}_p} & |v|_p^{-9s+4} |r|_p^{-6s+2}
    (1-|r|_p^{3s-1}p^{-3s+1}) \\ &
    \psi(v(r+c+bu-u^3-3uyr-y^3r^2)) \,du \,dy \,dv \,dr.  
\end{aligned}
 \end{equation*}

Applying Corollary \ref{cor1},
\begin{equation*}
    \mathbf{I}_{-+-+} = \frac{1-p^{-3s}}{1-p^{-3s+1}} \Int_{\mathbf{Z}_p} \Int_{\mathbf{Z}_p} \Int_{\mathbf{Z}_p} \mathop{\sum_{V=-\infty}^{-\text{ord}(r)-1}} |r|_p^{-6s+2} p^{-V(-9s+5)} (1-|r|_p^{3s-1}p^{-3s+1})
\end{equation*}

\begin{equation*}
     \Int_{\mathbf{Z}_p^*} \psi(\nu p^V(c+bu-u^3+r-3uyr-y^3r^2)) \,d\nu \,du \,dy \,dr.
\end{equation*}

Applying Lemma \ref{lem3},
\begin{equation*}
    \mathbf{I}_{-+-+} = \frac{1-p^{-3s}}{1-p^{-3s+1}} \Int_{\mathbf{Z}_p} \Int_{\mathbf{Z}_p} \Int_{\mathbf{Z}_p} \mathop{\sum_{V=-\infty}^{-\text{ord}(r)-1}} |r|_p^{-6s+2} p^{-V(-9s+5)} (1-|r|_p^{3s-1}p^{-3s+1})
\end{equation*}

\begin{equation*}
    (\mathbbm{1}_{\mathbf{Z}_p}(p^V(g(u)-h(u,y,r))) - p^{-1}\mathbbm{1}_{\mathbf{Z}_p}(p^{V+1}(g(u)-h(u,y,r))))\,du \,dy \,dr,
\end{equation*}

where $g(u) = -u^3 +bu +c$ and $h(u,y,r) = 3uyr + y^3r^2 -r$. There are two possibilities: either $r \in \mathbf{Z}_p^*$ or not. 

Consider first $r \notin \mathbf{Z}_p^*$. Then 
$r \in p\Z-p.$ Since $u,y \in \mathbf{Z}_p$ then $h(u,y,r) \in p\mathbf{Z}_p.$  However since $g(u) \in \mathbf{Z}_p^*$ then $g(u) - h(u,y,r) \in \mathbf{Z}_p^*$. Therefore $\mathbbm{1}_{\mathbf{Z}_p}(p^{-k}(g(u)-h(u,y,r))) = 0$, $\forall k>0$.

Now consider $r \in \mathbf{Z}_p^*$. Define 
\begin{equation*}
\begin{aligned}
N(V) &= \# \{(r,y,u) \in (\frac{\mathbf{Z}}{p^{-V}\mathbf{Z}})^* \times (\frac{\mathbf{Z}}{p^{-V}\mathbf{Z}}) \times (\frac{\mathbf{Z}}{p^{-V}\mathbf{Z}}) |  g(u) - h(r,y,u) \equiv 0 \mod p^{-V} \},\\
    H(V) &= \Int_{\mathbf{Z}_p} \Int_{\mathbf{Z}_p} \Int_{\mathbf{Z}_p^*} \mathbbm{1}_{\mathbf{Z}_p}(p^V(g(u)-h(u,y,r))) \,dr \,du \,dy\\
    &=\on{Vol}\{ (r,y,u) \in \mathbf{Z}_p^* \times \mathbf{Z}_p \times \mathbf{Z}_p | p^{-V}(g(u) - h(r,y,u))\in \mathbf{Z}_p\}.
    \end{aligned}
\end{equation*}

Then, $H(V) = p^{3V} N(V),$ and 
\begin{equation*}
    \mathbf{I}_{-+-+} = (1-p^{-3s}) \mathop{\sum_{V=-\infty}^{-1}} p^{-V(-9s+5)} (H(V) - p^{-1}H(V+1)).
\end{equation*}

\begin{conj}\label{conj1}
For all primes $p \equiv 2 \mod 3$, $N(-1) = p^2 -1$ ($\forall b,c$ in $\Z$ such that  $g$ is irreducible mod p).
\end{conj}

This restates conjecture \ref{conj0} from the introduction, and is proved by Scharaschkin in the appendix. 

\begin{cor} \label{cor3}
If $p \equiv 2 \mod 3$, k an integer, $k \leq 0$, then
$$
H(k) = Vol\{(r,y,u) \in \mathbf{Z}_p^* \times \mathbf{Z}_p \times \mathbf{Z}_p | p^{k}(g(u) - h(r,y,u))) \in \mathbf{Z}_p\} =
\left\{
   \begin{array}{cc}
      \frac{p-1}{p}, & \text{if } k=0,\\
      \frac{p^2-1}{p^3}(p^{k+1}), & \text{if } k \leq -1. \\
   \end{array}\right.
$$
\end{cor}

\begin{proof}
Let us show first that when $k=0,$ $H(k) = \frac{p-1}{p}$. Notice that $\forall p$, the volume of the set $\mathbf{Z}_p$ is $1$ and the volume of the set $\mathbf{Z}_p^*$ is simply $\frac{p-1}{p}$. Then the volume of the Cartesian product $\mathbf{Z}_p^* \times \mathbf{Z}_p \times \mathbf{Z}_p$ is just $(\frac{p-1}{p})(1)(1) = \frac{p-1}{p}$. Now we must show the case where $k \leq -1$. Notice that by conjecture \ref{conj1}, $H(k) = p^{3k}N(k)$ implies $H(-1) = p^{3(-1)}(p^2-1) = \frac{p^2-1}{p^3} = (\frac{p^2-1}{p^3})(p^{-1+1})$. So the base case where $k = -1$ is satisfied. Assume the statement is true $ \forall k \in \{-1,-2,...,k'\}$, i.e. $H(k) = p^{3k}N(k) = (\frac{p^2-1}{p^3})(p^{k+1})$. Then the next case by definition is, $H(k'-1) = p^{3(k'-1)}N(k'-1)$. By Corollary \ref{cor2}, $N(k-1) = p^2N(k), \forall k\leq -1$. Therefore,
\begin{equation*}
\begin{aligned}
    H(k'-1) &= p^{3(k'-1)}(p^2)N(k') = p^{3k'}p^{-3}p^2N(k') = p^{-1}(p^{3k'}N(k')) \\&= p^{-1}(\frac{p^2-1}{p^3})(p^{k'+1}) = (\frac{p^2-1}{p^3})(p^{(k'-1)+1}).
    \end{aligned}
\end{equation*}

Hence, case $k = k'-1$ is also satisfied. By the Principle of Strong Mathematical Induction, $\forall k \leq -1$, $H(k) = (\frac{p^2-1}{p^3})(p^{k+1})$. 
\end{proof}
Therefore assuming $p \equiv 2 \mod 3$, we can apply Corollary \ref{cor3}.

\begin{equation*}
H(V) -p^{-1}H(V+1) =
\left\{
   \begin{array}{cc}
      (\frac{p^2-1}{p^3}) - p^{-1}\frac{p-1}{p} = \frac{p-1}{p^3} ,& \text{if } V=-1,\\
      (\frac{p^2-1}{p^3})(p^{V+1}) - p^{-1}(\frac{p^2-1}{p^3})(p^{V+2})) = 0 , & \text{if } V < -1 .\\
   \end{array}\right.
\end{equation*}

Therefore,
\begin{equation*}
    \mathbf{I}_{-+-+} =(1-p^{-3s}) p^{-9s+5}(\frac{p-1}{p^3}) = (1-p^{-3s})(1-p^{-1})p^{-9s+3} = (1-p^{-3s})(p^{-9s+3}-p^{-9s+2}).
\end{equation*}

Before moving on to the next case, we remark that, even without conjecture
\ref{conj1} we can say that $H(k) = p^{-1}H(k+1)$ for $k < -1,$ 
and hence that
$$H(V) - p^{-1}H(V+1)=\begin{cases}
    \frac{N(-1)-p^2+p}{p^3}, & V=-1\\ 0, &V< -1,
\end{cases}$$
and 
$$\mathbf{I}_{-+-+}
= (1-p)^{-3s} p^{-9s+3}(1-p+\frac{N(-1)}p).
$$
Thus, conjecture \ref{conj1} can actually be deduced from Xiong's
result, by way of our results on the other fifteen sub-integrals. 

Case 12 ($I_{-+--}$) 

\begin{equation*}
\begin{aligned}
    \mathbf{I}_{-+--} =&
     \Int_{\mathbf{Q}_p - \mathbf{Z}_p} \Int_{\mathbf{Q}_p - \mathbf{Z}_p} \Int_{\mathbf{Q}_p - \mathbf{Z}_p} \Int_{\mathbf{Z}_p}   \psi (z+cv- u^3v+buv-3uyz -v^{-1}y^3z^2)\\&
|v|_p^{-3s+1} |z|_p^{-6s+2}  |y|_p^{-9s+3}
     \Int_{\mathbf{Q}_p} f_s(n^-(-x)) \psi(v^{-1}xy^3z) \,dx \,du \,dz \,dv \,dy.\end{aligned}
\end{equation*}

By Lemma \ref{lem1}, if $v^{-1}y^3z \notin \mathbf{Z}_p$ then the $x$ integral vanishes. Hence we obtain a new integral
where the domain of integration includes the constraint 
$v^{-1}y^3z \in \mathbf{Z}_p$. 
\begin{equation*}
    \mathbf{I}_{-+--} = \frac{1-p^{-3s}}{1-p^{-3s+1}}\Int_{\mathbf{Q}_p - \mathbf{Z}_p} \Int_{\mathbf{Q}_p - y^3\mathbf{Z}_p} \Int_{y^{-3} v \Z_p - \mathbf{Z}_p} \Int_{\mathbf{Z}_p}  |v|_p^{-3s+1} |z|^{-6s+2}  |y|_p^{-9s+3} 
\end{equation*}

\begin{equation*}
    \psi (z+cv- u^3v+buv-3uyz -v^{-1}y^3z^2 + v^{-1}xy^3z)(1-|\frac{y^{3}z}{v}|_p^{3s-1}p^{-3s+1}) \,du \,dz \,dv \,dy.
\end{equation*}

Let $z = rv$, then $dz = |v|_p dr.$ This implies that $|rv|_p > 1 \iff |r|_p>|v|_p^{-1}$, $|y^3r|_p \leq 1 \iff |r|_p \leq |y|_p^{-3}$. Therefore $|y|_p^3 < |v|_p$. Hence, 

\begin{equation*}
    \mathbf{I}_{-+--} = \frac{1-p^{-3s}}{1-p^{-3s+1}}\Int_{\mathbf{Q}_p - \mathbf{Z}_p} \Int_{|v|>|y|^3}\Int_{|v|^{-1}<|r| \leq |y|^{-3}} \Int_{\mathbf{Z}_p}  |v|_p^{4-9s} |r|_p^{2-6s}  |y|_p^{3-9s} (1-|y^3r|_p^{3s-1}p^{-3s+1}) 
\end{equation*}

\begin{equation*}
    \psi (v(r+c-u^3+bu-3uyr-y^3r^2)) \,du \,dr \,dv \,dy.
\end{equation*}

Notice that $|y^3r^2| \leq |r| \leq |y|^{-3} \leq p^{-3}$, which implies that $|yr| \leq |y||y|^{-3} \leq |y|^{-2} \leq p^{-2}$. Let $v = p^V\nu$ and $y = p^Y\gamma,$ with $\nu, \gamma \in \Z_p^*$. By Lemma \ref{lem2}, $dv = p^{-V}d\nu$ and $dy =p^{-Y} d\gamma$. Applying Corollary \ref{cor1},

\begin{equation*}
\begin{aligned}    \mathbf{I}_{-+--} = \mathop{\sum_{Y=-\infty}^{-1}\sum_{V=-\infty}^{3Y-1}} & p^{-V(5-9s)-Y(4-9s)} \Int_{p^V < |r|\leq p^{3Y}} |r|_p^{2-6s}
\\ &
     \Int_{\mathbf{Z}_p} \Int_{\mathbf{Z}_p^*} \Int_{\mathbf{Z}_p^*}
    \psi(p^V\nu(r+c-u^3+bu-3up^Y\gamma r-(p^Y\gamma)^3r^2)) \,d\nu \,d\gamma \,du \,dr.
\end{aligned}
\end{equation*}

By Lemma \ref{lem3},
\begin{equation*}
     \Int_{\mathbf{Z}_p^*} \psi(p^V\nu(r+c-u^3+bu-3up^Y\gamma r-(p^Y\gamma)^3r^2)) \,d\nu  
\end{equation*}

\begin{equation*}
    = \mathbbm{1}_{\mathbf{Z}_p}(p^V\nu(r+c-u^3+bu-3up^Y\gamma r-(p^Y\gamma)^3r^2))-p^{-1}\mathbbm{1}_{\mathbf{Z}_p}(p^{V+1}\nu(r+c-u^3+bu-3up^Y\gamma r-(p^Y\gamma)^3r^2)).
\end{equation*}

Assuming that $u^3- bu-c$ is irreducible mod p ($b,c \in \mathbf{Z}_p^*$) and $p>3$ makes $p^V\nu(r+c-u^3+bu-3up^Y\gamma r-(p^Y\gamma)^3r^2) \notin p^{-1}\mathbf{Z}_p$.  
Therefore $\mathbbm{1}_{\mathbf{Z}_p}(p^V\nu(r+c-u^3+bu-3up^Y\gamma r-(p^Y\gamma)^3r^2)) = 0$ and the integral vanishes. 

Case 13 ($I_{--++}$) 

\begin{equation*}
\mathbf{I}_{--++} 
     = \Int_{\mathbf{Q}_p - \mathbf{Z}_p} \Int_{\mathbf{Q}_p - \mathbf{Z}_p} \Int_{\mathbf{Z}_p} \Int_{\mathbf{Z}_p}   \psi (vg(u)- (2uz + 3y)u^2 - \frac{(u^2z^2 +3y^2 +3uyz)u}{v}) 
\end{equation*}

\begin{equation*}
     |v|_p^{-3s+1}  |u|_p^{-9s+4} \Int_{\mathbf{Q}_p} f_s(n^-(-x)) \psi(v^{-1}x) \,dx  \,dz \,dy \,du \,dv,
\end{equation*}
where $g(u) = c +bu -u^3$.

Notice that $v \notin \mathbf{Z}_p$ which implies that $ord(v) < 0,$
and hence that 
 $v^{-1} \in \mathbf{Z}_p$. By Lemma \ref{lem1},
 
\begin{equation*}\begin{aligned}
    \mathbf{I}_{--++} = & \frac{1-p^{-3s}}{1-p^{-3s+1}}\Int_{\mathbf{Q}_p - \mathbf{Z}_p} \Int_{\mathbf{Q}_p - \mathbf{Z}_p} \Int_{\mathbf{Z}_p} \Int_{\mathbf{Z}_p} |v|_p^{-3s+1}  |u|_p^{-9s+4}  
(1-|v^{-1}|^{3s-1}p^{-3s+1})\\&
    \psi (vg(u)- (2uz + 3y)u^2 - \frac{(u^2z^2 +3y^2 +3uyz)u}{v}) \,dz \,dy \,du \,dv.
    \end{aligned}
\end{equation*}

Let $y=v y_1$, $z= vz_1$. By Lemma \ref{lem2}, $\,dy = |v|_p\,dy_1$ and $\,dz = |v|_p\,dz_1$. 

\begin{equation*}\begin{aligned}
    \mathbf{I}_{--++} = & \frac{1-p^{-3s}}{1-p^{-3s+1}}\Int_{\mathbf{Q}_p - \mathbf{Z}_p}  \Int_{\mathbf{Q}_p - \mathbf{Z}_p} \Int_{v^{-1}\mathbf{Z}_p} \Int_{v^{-1}\mathbf{Z}_p} |v|_p^{-3s+3}  |u|_p^{-9s+4}  
(1-|v^{-1}|^{3s-1}p^{-3s+1})\\&
    \psi (v(g(u)- (2uz_1 + 3y_1)u^2 - ((uz_1)^2 +3y_1^2 +3uy_1z_1)u))\,dy_1 \,dz_1 \,du \,dv.
    \end{aligned}
\end{equation*}

\begin{lem}\label{lem8}
    For all $u,y_1,z_1$ in the domain of integration, the order of $g(u)- (2uz_1 + 3y_1)u^2 - ((uz_1)^2 +3y_1^2 +3uy_1z_1)u$ is equal to $3 \ord(u).$ In particular, it is at most $-3.$
\end{lem}

\begin{proof}
Since $\ord(u) < 0$ it follows from the strong triangle inequality that 
$\ord(g(u)) = \ord(u^3) = 3 \ord(u).$ We claim that this is the minimal 
order among the three summands. 

First, notice that since $y_1, z_1 \in v^{-1}\mathbf{Z}_p$, it is true that both $\ord(z_1) \text{ and } \ord(y_1)$ are positive, so $\ord(u)$ is strictly smaller to them. Using this, it is sufficient to show the order of the first summand is smaller than the other two. We will only show the case of the second summand, since the third follows analogously. Since $2$ and $3$  have zero order, then $\ord(2(uz_1)u^2) = \ord(u^3z_1) = 3\ord(u) + \ord(z_1) > 3\ord(u)$. Similarly, $\ord(3y_1u^2) = \ord(y_1u^2) = \ord(y_1) + 2\ord(u) \geq 2\ord(u) > 3\ord(u)$ (Again because the $\ord(u)$ is negative). Then $\ord((2uz_1+3y_1)u^2) = \ord(u^3z_1+y_1u^2) = \ord(u^3z_1) + \ord(y_1u^2) > 3\ord(u)$. Therefore, the order of the first summand is smaller than the orders of the other two, and that means $3\ord(u)$ is the order of the sum.\end{proof}

In view of Lemma \ref{lem8}, it follows from Corollary
\ref{cor1} and Lemma \ref{lem3} that $I_{--++}$ vanishes. 

Case 14 ($I_{--+-}$)

\begin{equation*}
    \mathbf{I}_{--+-} 
     = \Int_{\mathbf{Q}_p - \mathbf{Z}_p} \Int_{\mathbf{Q}_p - \mathbf{Z}_p} \Int_{\mathbf{Q}_p - \mathbf{Z}_p} \Int_{\mathbf{Z}_p}  |v|_p^{-3s+1} |y|_p^{-9s+3} |u|_p^{-9s+4} 
\end{equation*}

\begin{equation*}
   \psi (vg(u)- (2uz + 3y)u^2 - \frac{(u^2z^2 +3y^2 +3uyz)u}{v})  \Int_{\mathbf{Z}_p} f_s(n^-(-x)) \psi(v^{-1}xy^3) \,dx  \,dz \,dy \,du \,dv. 
\end{equation*}
where $g(u) = c +bu -u^3$.

If $v^{-1}y^3 \notin \mathbf{Z}_p$ the $x$ integral vanishes. Note that  $v^{-1}y^3 \in \mathbf{Z}_p$
is equivalent to 
$|y|_p^3 \le |v|$. By Lemma \ref{lem1}, 

\begin{equation*}
    \mathbf{I}_{--+-} =  \frac{1-p^{-3s}}{1-p^{-3s+1}}\int\limits_{\mathbf{Q}_p - \mathbf{Z}_p}\ \iint\limits_{|v|_p\ge |y|_p^3> 1} \ \Int_{\mathbf{Z}_p} |v|_p^{-3s+1} |y|_p^{-9s+3} |u|_p^{-9s+4} 
\end{equation*}

\begin{equation*}
    \psi (vg(u) - (2uz+3y)u^2 - \frac{(u^2z^2+3y^2+3uyz)u}{v}) (1-|v^{-1}y^3|^{3s-1}p^{-3s+1}) \,dz \,dv \,dy \,du.
\end{equation*}

Let $y=v y_1$, $z= vz_1$. Then $z_1 \in v^{-1}\mathbf{Z}_p,$ and $|v|_p^{-2} \ge |y_1|^3_p> |v|_p^{-3}.$ 
Also, by Lemma \ref{lem2}, $\,dy = |v|\,dy_1$ and $\,dz = |v| \,dz_1$.
Hence, 
\begin{equation*}
    \mathbf{I}_{--+-} =  \frac{1-p^{-3s}}{1-p^{-3s+1}}\Int_{\mathbf{Q}_p - \mathbf{Z}_p}\Int_{\mathbf{Q}_p - p\mathbf{Z}_p}\  
    \Int_{|v|_p^{-2} \ge |y_1|_p^3 > |v|_p^{-3}} 
    \ \Int_{v^{-1}\mathbf{Z}_p} |v|_p^{-3s+3} |y|_p^{-9s+4} |u|_p^{-9s+4} (1-|v^{-1}y^3|^{3s-1}p^{-3s+1})
\end{equation*}

\begin{equation*}
    \psi (v(g(u)- (2uz_1 + 3y_1)u^2 - ((uz_1)^2 +3y_1^2 +3uy_1z_1)u)) \,dz_1  \,dy_1 \,dv \,du.
\end{equation*}

Notice that, inside of $\psi$, 
we have the same function as in $I_{--++}$. Let's check if Lemma \ref{lem8} is still true. 
First, note that each of the 
variables $z_1,y_1$ ranges over a subset of 
$p\mathbf{Z}_p$ which depends on $v,$ 
while $u$ ranges over $\Q_p-\Z_p.$ 
Hence, $\ord(z_1) \geq 1 > \ord(u)$. Then $\ord(2(uz_1)u^2) = \ord(u^3z_1) = 3\ord(u) + \ord(z_1) > 3\ord(u)  \text{ and } \ord(3y_1u^2) = \ord(u^2y_1) = 2\ord(u) + \ord(y_1) > 2\ord(u) -2/3\ord(v) > 2\ord(u) > 3\ord(u)$ (Since the order of both $u,v$ is negative). So, $\ord((2uz_1+3y_1)u^2) \geq \min(\ord(u^3z_1),\ord(u^2y_1)) > 3\ord(u)$. An analogous argument can also be made for the third summand. Once again, the order of the first summand is smaller than the orders of the other two, so Lemma \ref{lem8} holds.

In view of Lemma \ref{lem8}, it follows from Corollary
\ref{cor1} and Lemma \ref{lem3} that $I_{--+-}$ vanishes.



Case 15 ($I_{---+}$)

\begin{equation*}
    \mathbf{I}_{---+} 
    = \Int_{\mathbf{Q}_p - \mathbf{Z}_p} \Int_{\mathbf{Q}_p - \mathbf{Z}_p} \Int_{\mathbf{Q}_p - \mathbf{Z}_p} \Int_{\mathbf{Z}_p} |v|_p^{-3s+1}  |u|_p^{-9s+4} |z|_p^{-6s+2}
\end{equation*}

\begin{equation*}
    \psi(cv- u^3v- 2u^3z +buv - u^3v^{-1}z^2 -3u^2yz - v^{-1}y^3z^2 -3u^2v^{-1}yz^2 -3uv^{-1}y^2z^2) 
\end{equation*}

\begin{equation*}
    \Int_{\mathbf{Q}_p} f_s(n^-(-x))\psi(v^{-1}zx) \,dx \,dy \,du \,dv \,dz.
\end{equation*}
By Lemma \ref{lem1}, and the fact that
$v^{-1} z \in \Z_p \iff z \in v \Z_p,$
\begin{equation*}
    \mathbf{I}_{---+} = \frac{1-p^{-3s}}{1-p^{-3s+1}} \Int_{\Q_p-\Z_p} \Int_{\Q_p-\Z_p} \Int_{v \Z_p - \Z_p}  \Int_{\mathbf{Z}_p} |v|_p^{-3s+1}  |u|_p^{-9s+4} |z|_p^{-6s+2}
\end{equation*}
\begin{equation*}
    \psi(cv- u^3v- 2u^3z +buv - u^3v^{-1}z^2 -3u^2yz - v^{-1}y^3z^2 -3u^2v^{-1}yz^2 -3uv^{-1}y^2z^2)(1-|v^{-1}z|^{3s-1}p^{-3s+1}) \,dy \,dz \,dv \,du.
\end{equation*}

Let $z = rv.$  Then by Lemma \ref{lem2}, $dz = |v|_p dr$. Rearranging the terms inside $\psi$, we get the following:

\begin{equation*}
    \mathbf{I}_{---+} =\frac{1-p^{-3s}}{1-p^{-3s+1}} \Int_{\Q_p-\Z_p} \iint_{1 \geq |r|_p > |v|_p^{-1}}   
    \Int_{\mathbf{Z}_p} (1-|r|_p^{3s-1}p^{-3s+1})|v|_p^{-9s+4} |u|_p^{-9s+4} |r|_p^{-6s+2}
\end{equation*}

\begin{equation*}
    \psi(v(-u^3(r+1)^2 -u^2(3yr(r+1)) + u(b-3y^2r^2) + (c-r^2y^3))) \,dy \,dr  \,dv   \,du. 
\end{equation*}

Since $y,r,b$ and $c$ are all in $\Z_p,$ it follows that 
$$\ord(-u^3(r+1)^2 -u^2(3yr(r+1)) + u(b-3y^2r^2) + (c-r^2y^3))
= \ord(-u^3(r+1)^2 ) \le -3,
$$
whenever $\ord(u(r+1)) < 0.$ On the other hand, 
if $\ord(u(r+1)) \ge 0,$ then $\ord(r+1) > 0,$
and hence $\ord(r) =1.$ The constraint
$|r|_p > |v|_p^{-1}$ becomes $v \in \Q_p-\Z_p$
for all such $r.$ We obtain 

\begin{equation*}
   \mathbf{I}_{---+}  =(1-p^{-3s}) \Int_{\Q_p-\Z_p} \Int_{\Q_p-\Z_p} \Int_{|r+1|_p \leq |u|_p^{-1}}   
    \Int_{\mathbf{Z}_p} |uv|_p^{-9s+4} 
\end{equation*}

\begin{equation*}
    \psi(v(-u^3(r+1)^2 -u^2(3yr(r+1)) + u(b-3y^2r^2) + (c-r^2y^3))) \,dy \,dr  \,dv   \,du.
\end{equation*}

Now let $y_1 = y/u$. Since $y \in \mathbf{Z}_p,$
then $|y_1|_p \leq |u|_p^{-1}$. Then by Lemma \ref{lem2}, $dy = |u|_p dy_1$. Rearranging again the terms inside $\psi$, we get the following:

\begin{equation*}
    \mathbf{I}_{---+} =(1-p^{-3s}) \Int_{\Q_p-\Z_p} \Int_{\Q_p-\Z_p} \Int_{|r+1|_p \leq |u|_p^{-1}}   
    \Int_{|y_1|_p \leq |u|_p^{-1}} |v|_p^{-9s+4} |u|_p^{-9s+5}
\end{equation*}

\begin{equation*}
    \psi(v(c+bu -u^3(1+r(3y_1+2)+r^2(y_1+1)^3))) \,dy_1 \,dr  \,dv   \,du. 
\end{equation*}

Let $v_1 = vu^3.$ Since $v \in \mathbf{Q}_p - \mathbf{Z}_p,$ then $ |v_1|_p > |u|_p^{3}$. Then by Lemma \ref{lem2}, $dv = |u|_p^{-3} dv_1$. So,

\begin{equation*}
    \mathbf{I}_{---+} =(1-p^{-3s}) \Int_{\Q_p-\Z_p} \Int_{|v_1|_p > |u|_p^{3}} \Int_{|r+1|_p \leq |u|_p^{-1}}  \Int_{|y_1|_p \leq |u|_p^{-1}} |v_1|_p^{-9s+4} |u|_p^{18s-10}
\end{equation*}

\begin{equation*}
    \psi(\frac{v_1}{u^3}(c+bu) - v_1(1+r(3y_1+2)+r^2(y_1+1)^3))\,dy_1 \,dr  \,dv_1   \,du.
\end{equation*}

Let $u_1 = u^{-1}.$ By Lemma \ref{lem2}, $\frac{du}{|u|_p} = \frac{du_1}{|u_1|_p},$ so $du = |u_1|_p^2du_1$. 
As $u$ ranges over $\Q_p-\Z_p,$ its inverse, 
$u_1,$ ranges over $p\Z_p-\{0\}.$
Since the measure of $\{0\}$ is $0,$ we obtain,

\begin{equation*}
    \mathbf{I}_{---+} =(1-p^{-3s}) \Int_{p\Z_p} \Int_{|v_1|_p > |u|_p^{3}} \Int_{|r+1|_p \leq |u|_p^{-1}}  \Int_{|y_1|_p \leq |u|_p^{-1}} |v_1|_p^{-9s+4} |u_1|_p^{-18s+8}
\end{equation*}

\begin{equation*}
    \psi(v_1(u_1^3c+u_1^2b) - v_1(1+r(3y_1+2)+r^2(y_1+1)^3)) \,dy_1 \,dr  \,dv_1   \,du_1 
\end{equation*}

\begin{equation*}
    =(1-p^{-3s}) \mathop{\sum_{U=1}^{\infty}} \Int_{|u_1|_p=p^{-U}} \Int_{|v_1|_p > p^{3U}} \Int_{|r+1|_p \leq p^{-U}}  \Int_{|y_1|_p \leq p^{-U}} |v_1u_1^2|_p^{-9s+4} 
\end{equation*}

\begin{equation*}
    \psi(v_1(u_1^3c+u_1^2b) - v_1(1+r(3y_1+2)+r^2(y_1+1)^3)) \,dy_1 \,dr  \,dv_1   \,du_1.
\end{equation*}

Let $u_1 = p^U u_2, u_2 \in \mathbf{Z}_p^*$. Then by Lemma 2, $du_1 = p^{-U}du_2$. Applying Corollary \ref{cor1},

\begin{equation*}
   =(1-p^{-3s}) \mathop{\sum_{U=1}^{\infty}} p^{U(18s-9)} \Int_{\mathbf{Z}_p^*} \psi(v_1(u_2^3p^{3U}c+u_2^2p^{2U}b)\,du_2   
\end{equation*}

\begin{equation*}
    \Int_{|v_1|_p > p^{3U}}|v_1|_p^{-9s+4} \Int_{|r+1|_p \leq p^{-U}}  \Int_{|y_1|_p \leq p^{-U}}\psi(-v_1(1+r(3y_1+2)+r^2(y_1+1)^3)) \,dy_1 \,dr \,dv_1.
\end{equation*}

We will focus on the $u_2$ integral and show it vanishes. For $a,d \in \Q_p,$ define $H(d)$ and $H_1(d,a)$ as the following:
\begin{equation*}
 H_1(d,a) = \Int_{\mathbf{Z}_p^*} \psi(d(x^2+ax^3)) \,dx,\qquad  
 H(d) = H_1(d,0)= \Int_{\mathbf{Z}_p^*} \psi(dx^2) \,dx .
\end{equation*}


\begin{lem} \label{lem9}
If $a \in p\Z_p$, then $H_1(d,a)=H(d),$
and if $\on{ord}(d) < -1,$ then $H(d)=0.$
\end{lem}
This can be proven by fixing $x_0 \in \{1,...,p-1\}$ and considering $\varphi_a = x^2 +ax^3$. Notice that $\varphi_a: \mathbf{Z}_p \rightarrow \mathbf{Z}_p,$
for $a \in \Z_p.$

\begin{lem} \label{cor4} 
For $a\in p \Z_p,$
the function 
$\varphi_a$ is a bijection $$\{x \in (\frac{\mathbf{Z}}{p^k\mathbf{Z}})^*| x \equiv x_0 \mod p\} \rightarrow \{y \in (\frac{\mathbf{Z}}{p^k\mathbf{Z}})^*| y \equiv \varphi_a(x_0) \mod p\},$$
for any positive integer $k,$ 
and hence a measure-preserving 
bijection 
$x_0 + p \Z_p \to x_0^2 + p \Z_p,$
for each $x_0 \in \Z_p^*.$
\end{lem}
\begin{proof}
    We will prove this statement by induction. If $k = 1$ then we have nothing to prove. Take $k \geq 1$ and write $x = x_1 + p^kx_2$. Then $\varphi_a(x)\equiv \varphi_a(x_1) + p^k(2x_1x_2)\mod{p^{k+1}}.$
    Since $p \ne 2$ and $x_1 \in (\Z/p\Z)^*$ the mapping $x_2 \mapsto 2x_1 x_2$ is a bijection $\Z/p\Z \to \Z/p\Z.$ It follows that 
    $$x_1 + p^kx_2 \mapsto \varphi_a(x_1) + p^k(2x_1x_2)
    $$
    is a bijection 
    $$\{ x \in (\Z/p^{k+1}\Z)^*: x \equiv x_1 \mod{p^k} \} \to \{ y \in (\Z/p^{k+1}\Z)^*: y \equiv \varphi_a(x_1) \mod p^k\}.
    $$
    Thus our bijection at level $k$ extends to a bijection at level $k+1,$
    completing the proof by induction.
\end{proof}
We can now prove Lemma \ref{lem9}. 

\begin{proof}[Proof of Lemma \ref{lem9}]
\begin{equation*} 
    H_1(d,a) = \Int_{\mathbf{Z}_p^*} \psi(d\varphi_a(x)) \,dx = \mathop{\sum_{x_0=1}^{p-1}} \Int_{x_0 + p\mathbf{Z}_p} \psi(d\varphi_a(x)) \,dx = \mathop{\sum_{x_0=1}^{p-1}} \Int_{\varphi_a(x_0) + p\mathbf{Z}_p} \psi(dy) \,dy \text{, Corollary \ref{cor4}}
\end{equation*}

Notice that $\varphi_a(x_0) +p\mathbf{Z}_p = x_0^2 + pax_0^3 +p\mathbf{Z}_p = x_0^2 + p\mathbf{Z}_p$

\begin{equation*}
    H_1(d,a) =  \mathop{\sum_{x_0=1}^{p-1}} \Int_{x_0^2 + p\mathbf{Z}_p} \psi(dy) \,dy = H_1(d,0) = H(d). 
\end{equation*}
Moreover, for each $x_0$
we have 
$$
\Int_{\varphi_a(x_0) + p\mathbf{Z}_p} \psi(dy) \,dy
= p^{-1}\Int_{\Z_p}\psi(d(\varphi_a(x_0) + p y_1)) \, dy_1
= p^{-1}\psi(d\varphi_a(x_0) \Int_{\Z_p}\psi(dp y_1) \, dy_1,
$$
 by lemma \ref{lem2}.
 If $\ord(d) < -1,$ then the last integral is 
 zero by lemma \ref{lem:char(a)}.
\end{proof}
Coming back to the $u_2$ integral,

\begin{equation*}
\Int_{\mathbf{Z}_p^*} \psi(v_1(u_2^3p^{3U}c+u_2^2p^{2U}b)\,du_2 = H_1(v_1p^2Ub,p^Ucb^{-1}) \text{, provided $b \neq 0$}   
\end{equation*}

We know $u>0$, so $p^Ucb^{-1}$ is always divisible by p. By Lemma $\ref{lem9}$, $H_1(v_1p^2Ub,p^Ucb^{-1}) = H(v_1p^{2U}b)$. Since $|v_1|_p>p^{3U}$ and $b \in \mathbf{Z}_p^* \rightarrow |v_1p^{2U}b|_p = |v_1p^{2U}|_p > p^{3U}p^{-2U} = p^{U} > 1$. Therefore, $\ord(v_1p^{2U}b) < -1$ implies $H(v_1p^{2U}b) = 0$ and the integral vanishes. 

$\textbf{Note}$:  Jiang-Rallis, showed that $\mathbf{I}_{---+}$ has a nonzero solution in the case $b=0.$ 

Case 16 ($I_{----}$)

\begin{equation*}
    \mathbf{I}_{----} 
    = \Int_{\mathbf{Q}_p - \mathbf{Z}_p} \Int_{\mathbf{Q}_p - \mathbf{Z}_p} \Int_{\mathbf{Q}_p - \mathbf{Z}_p} \Int_{\mathbf{Q}_p - \mathbf{Z}_p} |v|_p^{-3s+1}  |u|_p^{-9s+4} |z|_p^{-6s+2} |y|^{-9s+3} 
\end{equation*}

\begin{equation*}
    \psi(vg(u) - (2u+3y)u^2z -\frac{(u+y)^3 z^2}v) 
\end{equation*}

\begin{equation*}
    \Int_{\mathbf{Q}_p} f_s(n^-(-x))\psi(v^{-1}xy^3z) \,dx \,du \,dv \,dy \,dz
\end{equation*}
where $g(u) = c+bu-u^3.$

If $v^{-1}y^3z \notin \mathbf{Z}_p$ the $x$ integral vanishes. 
If not then $z$ lies in 
$v^{-1}y^3 \Z_p-\Z_p.$ Notice that this is only nonempty  
if $3 \ord(y) > \ord(v),$ i.e. $v \in \Q_p -y^3 \Z_p.$
By Lemma \ref{lem1}, 
\begin{equation*}
    \mathbf{I}_{----} = \frac{1-p^{-3s}}{1-p^{-3s+1}} \Int_{\mathbf{Q}_p - \mathbf{Z}_p} \Int_{\mathbf{Q}_p - \mathbf{Z}_p} \Int_{\mathbf{Q}_p - y^3\mathbf{Z}_p} \Int_{\frac{v}{y^3}\mathbf{Z}_p - \mathbf{Z}_p} |v|_p^{-3s+1}  |u|_p^{-9s+4} |z|_p^{-6s+2} |y|^{-9s+3} 
    (1-|v^{-1}y^3 z|^{3s-1}p^{-3s+1})
\end{equation*}
\begin{equation*}
     \psi(vg(u) - (2u+3y)u^2z -\frac{(u+y)^3 z^2}v)  \,dz \,dv \,dy \,du.
\end{equation*}
Let $z = vr.$ By Lemma \ref{lem2}, $dz = |v| dr,$
and
\begin{equation*}
    \mathbf{I}_{----} = \frac{1-p^{-3s}}{1-p^{-3s+1}} \Int_{\mathbf{Q}_p - \mathbf{Z}_p} \Int_{\mathbf{Q}_p - \mathbf{Z}_p} \Int_{\mathbf{Q}_p - y^3\mathbf{Z}_p} \Int_{y^{-3}\mathbf{Z}_p - v^{-1}\mathbf{Z}_p} |v|_p^{-9s+4}  |u|_p^{-9s+4} |r|_p^{-6s+2} |y|^{-9s+3} 
    (1-|y^3 r|^{3s-1}p^{-3s+1})
\end{equation*}
\begin{equation*}
     \psi(v[g(u) - (2u+3y)u^2r -(u+y)^3 r^2])  \,dr \,dv \,dy \,du.
\end{equation*}

\begin{lem}\label{lem:minord}
    For all $u,y,r$ in the domain of integration, the order of $g(u) - (2u+3y)u^2r -(u+y)^3 r^2$ is equal to 
    $3 \ord(u).$ In particular, it is at most $-3.$
\end{lem}
\begin{proof}
Since $\ord(u) < 0$ it follows from the strong triangle inequality that 
$\ord(g(u)) = \ord(u^3) = 3 \ord(u).$ We claim that this is the minimal 
order among the three summands. 

    First suppose that $\ord(u) \le \ord(y).$  Then $\ord(u^2(3u+3y))$
and $\ord(u+y)^3$ are at least $\ord(u^3).$ But $\ord(r) \ge \ord(y^{-3})
\ge 3,$ so $\ord(g(u))$ is indeed strictly smaller than the other 
two, and hence the order to the sum.

Now suppose that $\ord(y) < \ord(u).$ Then $\ord(u+y) = \ord(2u+3y) = \ord(y).$ Then 
$$
\ord((2u+3y) u^2 r) = \ord(yu^2 r) > \ord(y^3r) \ge 0 > \ord(u^3),
$$
and 
$$\ord((u+y)^3 r^2) 
= \ord(y^3 r^2) \ge \ord(r) > 0 > \ord(u^3).$$
Once again, the order of the first summand is strictly smaller than 
the orders of the other two, and that means it is the order of the 
sum.
\end{proof}
In view of Lemma \ref{lem:minord}, it follows from corollary
\ref{cor1} and Lemma \ref{lem3} that $I_{----}$ vanishes. 

\section{Summary of Results}

Now that we have the result of all 16 cases we can deduce the following propositions.


\begin{prop}\label{prop1}

$$I_{+++-} = I_{++-+} = I_{++--} = I_{+-+-} = I_{+--+} = I_{+---} = 0,$$

$$I_{--++} = I_{--+-} = I_{----} = 0 \text{. (New Result)}$$
\end{prop}

\begin{prop}\label{prop2}
If $g$ is irreducible mod p, $p > 3$, then $I_{-++-} = I_{-+--} = 0.$
\end{prop}
\begin{prop}\label{prop3} If $g$ is irreducible mod p, $p \equiv 2 \mod 3$, then 
$I_{---+} = 0 \text{. (New Result)}$
\end{prop}

\begin{prop}\label{prop4}

$I_{+} = (1 - p^{-3s})(1+p^{-9s+2}).$ 
\end{prop}
\begin{proof}We know that $I_{+} = I_{++++} + I_{+-++} + I_{+++-} + I_{++-+} + I_{++--} + I_{+--+} + I_{+-+-}  + I_{+---}$. By Proposition \ref{prop1}, $I_{+} = I_{++++} + I_{+-++}$. Then,

$I_{+} = (1 - p^{-3s})+(1-p^{-3s})p^{-9s+2} = (1 - p^{-3s})(1+p^{-9s+2}). $ 
\end{proof}
\begin{prop}\label{prop5}
If $g$ is irreducible mod $p,$ then\\ $I_{-} = (1-p^{-3s})(-p^{-3s+1}-p^{-6s+2}+p^{-9s+3}-p^{-9s+2}).$ (New result)
\end{prop}
\begin{proof}
We know that $I_{-} = I_{----} + I_{-+--} + I_{---+} + I_{--+-} + I_{--++} + I_{-++-} + I_{-+-+}  + I_{-+++}$. By Proposition \ref{prop1} and \ref{prop3}, $I_{-} = I_{-+++} + I_{-+-+}$. Assuming $g$ is irreducible mod $p$ makes $N(b,c) = 0$ in $I_{-+++}$. Therefore,
$I_{-} = (1 - p^{-3s})(-p^{-3s+1}-p^{-6s+2}) + (1-p^{-3s})(p^{-9s+3}-p^{-9s+2}) = (1-p^{-3s})(-p^{-3s+1}-p^{-6s+2}+p^{-9s+3}-p^{-9s+2}).$ 
\end{proof}
\begin{thm}[Main Theorem]
If Conjecture \ref{conj1} is valid in $I_{-+-+}$, then
$$I^{\sigma}_p(s, f_{s,p}^\circ) = (1-p^{-3s})(1-p^{-3s+1})(1-p^{-6s+2}).$$
\end{thm}
\begin{proof}
$I^{\sigma}_p(s, f_{s,p}^\circ) = I_{+} + I_{-} = (1 - p^{-3s})(1+p^{-9s+2}-p^{-3s+1}-p^{-6s+2}+p^{-9s+3}-p^{-9s+2}) = (1-p^{-3s})((1-p^{-6s+2})-p^{3s+1}(1-p^{-6s+2})) = (1-p^{-3s})(1-p^{-3s+1})(1-p^{-6s+2}).$ 
\end{proof}

\section{Directions for future work}
There are a number of natural directions for future work. 
As we have mentioned, Conjecture \ref{conj0} may actually be deduced
from the result of Xiong and the results of this paper. But, the elementary
nature of the statement suggests that it should have an elementary proof which, together with the results of this paper, would give a second 
proof of a special case of Xiong's result. Thus, giving a proof of 
Conjecture \ref{conj0} which does not depend on Xiong's result is a 
natural direction. 

Another natural direction is to generalize the result to the 
context of an arbitrary non-Archimedean local field $F$ of characteristic
zero, which contains only one cube root of $1.$  

One could also try to relax the assumption that all data is unramified. 
For example, we have assumed that $b$ and $c$ are in $\Z_p^*$ because, 
for any $b,c \in \Q^*,$ the set of primes $p$ such that $b$ and $c$ 
are not in $\Z_p^*$ is finite. Nevertheless, if we hope to get the 
correct Euler factor at every place, we must handle the finite set of 
exceptional primes.

As noted
before, the local integral 
$I_{v}^\sigma( s, f_{s,v})$ of Jiang and Rallis was defined in 
the context of an arbitrary local field of characteristic zero. 
This includes the two Archimedean local fields, $\R$ and $\C.$ As far as we know the Archimedean zeta integrals have not been considered. One would 
expect that they would be related to the Gamma factors which appear
in the functional equations of the zeta functions obtained from the
non-Archimedean contributions. 

Finally, it is worth noting as Jiang and Rallis did, that the exceptional groups $F_4$ and $E_8$ have unipotent subgroups, analogous to our group $N$, but with field extensions of degree $4$ (in the $F_4$ case)or $5$ (in the $E_8$ case) playing the role that field extensions of degree $3$ have played here. The quintic case is particularly tantalizing because this would be the first time an extension with a non-solvable Galois group could appear. It is also noteworthy that Xiong's approach is based on the ``Exceptional $\Theta$ correspondence'' of Magaard-Savin \cite{Magaard-Savin}, and an embedding $G_2 \times PGL_3$ into the exceptional group 
$E_6.$ It is reasonable to hope that it may extend to the $F_4$ case. The exceptional $\Theta$ correspondence 
also includes an  embedding of $G_2 \times F_4$
into $E_8$ \cite{Magaard-Savin}. But it's not clear how to embed a product of the form $E_8 \times H$ into something even bigger in a similar way. With that being said, it must be noted that, if $N'$ and $N''$ are the analogues of $N$ in $F_4$ and $E_8$ respectively, then $\dim N' = 20$ and $\dim N''=104.$ Thus, the method of repeated bifurcation which produced $16=2^4$ sub-integrals from an 
integral over $5$ dimensional $N,$ would result in 
$2^{19}$ or $2^{103}$ sub-integrals in these cases. Thus, any 
feasible approach along these lines would require automating the 
analysis of most of the integrals. 
 
\section{Acknowledgement}

We would like to acknowledge the Alfred P. Sloan Foundation (SLOAN) for funding the project initially as a summer research in 2023 (G2021-17076).

\printbibliography 
\section{Appendix: Proof of Conjecture 2}
by Victor Scharaschkin,


\renewcommand{\baselinestretch}{1.15}
\setlength{\parindent}{0mm} 
\setlength{\parskip}{3mm}

\renewcommand{\theenumi}{\alph{enumi}}
\renewcommand{\labelenumi}{(\theenumi)}
\renewcommand{\theenumii}{\roman{enumii}}
\renewcommand{\labelenumii}{(\theenumii)}
\theoremstyle{definition} 
\newtheorem{Thm}{Theorem} 
\newcommand{\N}{\mathbb{N}} \newcommand{\No}{\mathbb{N}_0} 
\renewcommand{\Z}{\mathbb{Z}}
\renewcommand{\Q}{\mathbb{Q}}
 \renewcommand{\R}{\mathbb{R}}
\renewcommand{\C}{\mathbb{C}}
\newcommand{\Proof}{\noindent\textbf{Proof}\quad}
\newcommand{\forwards}{$(\Longrightarrow)\quad$} %
\newcommand{\backwards}{$(\Longleftarrow) \quad$}
\newcommand{\lhs}{\textsc{lhs}}
\newcommand{\rhs}{\textsc{rhs}}
\newcommand{\from}{\colon}
\newcommand{\legendre}[2]{\Bigl(\begin{array}{c} \!\!\!\frac{#1}{#2}\!\!\! \end{array} \Bigr)}
\newcommand{\vs}{\vphantom{\int^2}}
\newcommand{\IFF}{\textsc{iff}}

\newcommand{\Fp}{\mathbb{F}_p}
\newcommand{\Fpx}{\mathbb{F}_p^{\times}}
\newcommand{\FF}{\mathbb{F}}

\newcommand{\fon}{f_1}

We prove Conjecture~2.  See also Conjecture~34 and the preceding text for an explanation of how the polynomial~$\fon$ arises.

We make a slight change of notation: put $r=v$ and $y=w$.

Let $p\equiv 5\pmod{6}$ be prime, and let $b$, $c \in \Fp$.  Define
\[
g(z) = z^3-bz-c,\qquad \fon(u,v,w) = v^2w^3 +3uvw +u^3 -bu -v -c
\]
\[
X_1 =\bigl\{(u,v,w) \mid \fon(u,v,w)=0\bigr\}. %
\]
Assume $g$ is irreducible over~$\Fp$ from now on.  We prove:
\begin{thm}\label{Thm-Conj2}
The affine variety $X_1$ is a rational surface, and $|X_1(\Fp)| =p^2-1$.
\end{thm}

Observe that as immediate consequences of our assumptions on $p$ and $g$,\vspace{-3mm}
\begin{enumerate}
\item[(a)] $-3$ is not a square in $\Fp$.\vspace{1mm}
\item[(b)] The cubing map on $\Fp$ is a bijection.\vspace{1mm}
\item[(c)] $g(a)\neq 0$ for any $a\in \Fp$.\vspace{1mm}
\item[(d)] $b \neq 0$.
\end{enumerate}\vspace{-3mm}
The condition $p\equiv 5\pmod{6}$ is used in (a) and (b).  Part (d) follows from (b) since $x^3-c$ is always reducible.  

We give an elementary proof of Theorem~\ref{Thm-Conj2}.  This is a little ungainly, so collect the notation here for reference.  Let

\begin{eqnarray*}
\fon(u,v,w) &=& v^2w^3 +3uvw +u^3 -bu -v -c \\[-1mm] 
f_2(u,v,w) &=& w \\[-1mm] %
f_3(u,v,w) &=& \textstyle\frac{1}{v}\!\Bigl[ g\bigl(u +vw^2\,\bigr) -\fon\Bigr] \;\;=\;\; v^2w^6 +3uvw^4 +3u^2w^2 -vw^3 -3uw -bw^2 +1 \\[-1mm] 
f_4(u,v,w) &=& \;w^3\!\cdot g\bigl(u \textstyle-\frac{1}{w}\bigr) \;\;=\;\; u^3w^3 -3u^2w^2 -buw^3 -cw^3 +3uw +bw^2 -1 \\[3mm] %
h_1(x,y) &=& 3x^2 +y^2 -4b \\[-1mm] %
h_2(x,y) &=& -x^3 +xy^2 +4c \\[-1mm] %
h_3(x,y) &=& 8\,g\bigl(\textstyle \frac{y-x}{2}\bigr) =(y-x)^3 -4b(y-x) -8c \\ %
h_4(x,y) &=& -8\,g\bigl(\textstyle \frac{-y-x}{2}\bigr) =(x+y)^3 -4b(x+y) +8c %
\end{eqnarray*}\vspace{-5mm}

Check the following identities:
\begin{eqnarray}
4g(x) +h_2 &=& x\,h_1 \label{eqxh1} \\ %
f_3+f_4 &=& w^3\fon \label{eqf3f4}\\ %
bh_1^2h_2 +ch_1^3 -h_2^3 &=& -g(x)h_3h_4 \label{eqh1h2h3} \\ %
8g(x) +6h_2 +h_3  &=& h_4 \label{eqh4}\\ %
4g(x) +3h_2 + h_3&=& y\,h_1 \label{eqyh1}\\ %
g\bigl(u +vw^2\,\bigr) &=&  \fon +v\cdot f_3 \label{eqgf1vf3}\\ %
\textstyle w^2\!\cdot h_1\bigl(u-\frac{1}{w},\; 2vw^2 +3u -\frac{1}{w}\bigr) &=& 4f_3 \label{eqh1theta}  \\ %
\textstyle w^2\!\cdot h_2\bigl(u-\frac{1}{w},\; 2vw^2 +3u -\frac{1}{w}\bigr) &=& 4(uf_3 -w^2\fon) \label{eqw2h2} \\ %
\textstyle w^3\!\cdot g\bigl(u -\frac{1}{w}\bigr) &=& f_4 \label{eqw3g} %
\end{eqnarray}

For each $f_j$ and $h_j$ above, define sets (affine varieties)
\begin{eqnarray*}
X_j &=& \{(u,v,w) \in \Fp^3 \mid f_j(u,v,w) =0\} \\ %
X_{j,k} &=& X_j \cap X_k \\ %
Y_j &=& \{(x,y) \in \Fp^2 \mid h_j(x,y) =0\} %
\end{eqnarray*}

Since $f_3(u,v,0)=1$ we have
\[
X_{2,3}(\Fp)= \varnothing.
\]
Since $g$ never vanishes, equations~(\ref{eqgf1vf3}) and~(\ref{eqw3g}) show
\begin{equation}\label{eq-X13}
X_{1,3}(\Fp) =\varnothing = X_{1,4}(\Fp).
\end{equation}
also.

We shall define a bijection
\begin{equation}\label{eq-bijection}
X_1 \setminus X_{1,2} \to \Fp^2 \setminus Y_1.
\end{equation}
Equation~(\ref{eq-bijection}) implies
\begin{equation}\label{eq-Xinitial}
|X_1(\Fp)| = p^2 -|Y_1(\Fp)| +|X_{1,2}(\Fp)|.
\end{equation}

Clearly $\fon(u,v,w)=f_2(u,v,w)=0 \iff g(u)=v$, so for each~$u$ there is a unique~$v$ with $(u,v,w) \in X_{1,2}$, so
\[
|X_{1,2}|=p.
\]

Since $b\neq 0$, $Y_1$ is a (non-degenerate) conic.  The projective closure of $Y_1$ has $p+1$ points over~$\Fp$. (It always has at least one point by the Cauchy--Davenport theorem,\footnote{Let $Q=\{3x^2 \mid x \in \Fp\}$, $R=\{y^2 \mid y \in \Fp\}$.  Then $|Q|+|R|\geq p+1$ implies every element of $\Fp$ (in particular $4b$) is of the form $q+r$ for some $q\in Q$, $r\in R$.} and this gives a bijection with $\mathbb{P}^1$.)  Points at $\infty$ would only occur where $3X^2 +Y^2 -4bZ^2=0$ with $Z=0$, that is, when $Y^2=-3X^2$.  So condition (a) above implies
\[
|Y_1|=p+1
\]
also.  Thus from~(\ref{eq-Xinitial}) we obtain
\[
|X_1|  = p^2 -p-1 +p = p^2-1.
\]
It remains to establish~(\ref{eq-bijection}).

Define $\theta \from \Fp^3\setminus X_2 \to \Fp^2$ by
\begin{equation}\label{eq-deftheta}
\theta(u,v,w) = \Bigl(u \textstyle-\frac{1}{w},\; 2vw^2 +3u -\frac{1}{w}\Bigr).
\end{equation}
From~(\ref{eqh1theta})\quad  [$w\neq 0$ and $f_3(u,v,w)\neq 0$\,] $\implies$ [\;$\theta(u,v,w)$ is defined and $h_1\bigl( \theta(u,v,w )\bigr)\neq 0$\,].  Thus
\[
\theta \from \Fp^3\setminus (X_2 \cup X_3) \to \Fp^2 \setminus Y_1.
\]
From~$(\ref{eq-X13})$ $X_{1,3}=\varnothing$, so $\theta$ restricts to a map
\begin{equation}\label{eq-theta-res}
\theta \from X_1 \setminus X_{1,2} \to \Fp^2 \setminus Y_1.
\end{equation}

Define $\sigma \from \Fp^2 \setminus Y_1 \to \Fp^3 \setminus X_2$ by
\begin{equation}\label{eq-defsigma}
\sigma(x,y) = \biggl(\;\frac{h_2}{h_1},\; \frac{8g(x)^2h_3}{h_1^3}, \; -\frac{h_1}{4g(x)}\;\biggr).
\end{equation}
This is well defined off $Y_1$ since $g$ never vanishes, and $h_1 \neq 0$ implies $\sigma(x,y) \not\in X_2$.  We show that $(x,y) \in \Fp^2\setminus Y_1 \implies \fon \circ \sigma(x,y) =0$ so $\sigma(x,y) \in X_1$.  Writing $g$ for $g(x)$,
\begin{eqnarray*}
-h_1^3 \!\cdot (\fon\circ \sigma) &=& -h_1^3\cdot f_1\Bigl(\frac{h_2}{h_1}, \frac{8g^2h_3}{h_1^3\vs}, -\frac{h_1}{4g}\Bigr) \\ &=& -h_1^3 \cdot \frac{8g^2h_3}{h_1^3\vs}\biggl[ \frac{-h_3-6h_2-8g}{8g \vs}\biggr] -h_2^3+bh_1^2h_2+ch_1^3 \\ &=& gh_3\bigl[ h_3+6h_2+8g\bigr] -h_2^3+bh_1^2h_2+ch_1^3 \\  &\stackrel{(\ref{eqh4}),\; (\ref{eqh1h2h3})}{=}& 0.
\end{eqnarray*}

So
\begin{equation}\label{eq-sigma-res}
\sigma \from \Fp^2\setminus Y_1 \to X_1 \setminus X_{1,2}.
\end{equation}

Finally we show that the maps $\theta$ and $\sigma$ in equations~(\ref{eq-theta-res}), (\ref{eq-sigma-res}) are mutually inverse:
\[
\theta\bigl(\sigma(x,y)\bigr)= \theta\! \left(\frac{h_2}{h_1\vs},\;\; \frac{8g^2h_3}{h_1^3\vs},\;\; -\frac{h_1}{4g(x)\vs} \right)  \;=\; \Bigl(\frac{4g+h_2}{h_1\vs},\;\; \frac{4g +3h_2 +h_3}{h_1 \vs} \Bigr) \;\stackrel{(\ref{eqxh1},\, \ref{eqyh1})}{=}\;  \bigl(x,\; y\bigr).
\]

Let $(U,V,W) =\bigl(\sigma\circ \theta\bigr)(u,v,w) = \sigma\bigl(u -\frac{1}{w},\; 2vw^2 +3u -\frac{1}{w}\bigr)$.  Here $(u,v,w) \in X_1\setminus X_{1,2}$ so $w \neq 0$ and $f_3(u,v,w)$ and $f_4(u,v,w)\neq 0$ also, from~(\ref{eq-X13}).  A (rather tedious) calculation yields:
\[
U = \frac{h_2}{h_1}\Bigl(u \textstyle-\frac{1}{w},\;\; 2vw^2 +3u -\frac{1}{w}\Bigr)\displaystyle \;\stackrel{(\ref{eqh1theta},\, \ref{eqw2h2})}{=}\; \frac{uf_3 -w^2\fon}{f_3} = u -\fon\frac{w^2}{f_3},
\]
\begin{multline*}
V=\frac{8g^2h_3}{h_1^3}\Bigl(u \textstyle -\frac{1}{w},\;\; 2vw^2 +3u -\frac{1}{w}\Bigr)\displaystyle \;\stackrel{(\ref{eqgf1vf3}),\, (\ref{eqh1theta}),\, (\ref{eqw3g})}{=}\; \frac{\fon f_4^2 +vf_3f_4^2}{f_3^3} \\[2mm] %
=\; v+ \frac{\fon f_4^2 +vf_3f_4^2 -vf_3^3}{f_3^3} \;\;\stackrel{(\ref{eqf3f4})}{=}\;\; v +\fon\frac{[f_4^2 +vw^3f_3\cdot (f_4 -f_3)]}{f_3^3},
\end{multline*}
and
\[
W=-\frac{h_1}{4g}\Bigl(u \textstyle -\frac{1}{w},\; 2vw^2 +3u -\frac{1}{w}\Bigr) \displaystyle \;\stackrel{(\ref{eqh1theta}),\, (\ref{eqw3g})}{=}\; -\frac{wf_3}{f_4}\;\; =\;\; w\;-\;w \cdot\! \left(1 +\frac{f_3}{f_4}\right) 
 \;\stackrel{(\ref{eqf3f4})}{=}\; w-\fon\frac{w^4}{f_4}.
\]

So
\[
(U,V,W) = \left(\!u-\fon\frac{w^2\fon}{f_3},\quad v+\fon\frac{[f_4^2 +vw^3f_3\cdot (f_4-f_3)]}{f_3^3},\quad w-\fon\frac{w^4}{f_4}\,\right).
\]
On $X_1$ where $\fon$ vanishes, $\bigl(\sigma\circ \theta\bigr)(u,v,w) =(U,V,W) =(u,v,w)$ as required.  This completes the proof.


\end{document}